\documentclass[11pt]{amsart}
\usepackage{amscd, amssymb, mathrsfs, mathabx, 
dsfont, mathtools, 
amsmath, soul,lipsum}
\usepackage{collectbox}
\usepackage[dvipsnames]{xcolor}
\usepackage[colorlinks, pagebackref]{hyperref}
\hypersetup{
    colorlinks=true,
    citecolor=Brown,
    linkcolor=MidnightBlue,
    urlcolor=Aquamarine}


\makeatletter
\@namedef{subjclassname@2020}{\textup{2020} Mathematics Subject Classification}
\makeatother

\input{xy}
\xyoption{all}

\newtheorem{Thm}{Theorem}[section]

\newtheorem{Prop}[Thm]{Proposition}
\theoremstyle{definition}
\newtheorem{Def}[Thm]{Definition}
\newtheorem{Def/Thm}[Thm]{Definition/Theorem}
\newtheorem{Cor}[Thm]{Corollary}
\newtheorem{lemma}[Thm]{Lemma}

\theoremstyle{remark}
\newtheorem{Rmk}[Thm]{Remark}

\newtheorem{EG}[Thm]{Example}
\numberwithin{equation}{section}
\newcommand{\ti }{\times}
\newcommand{\ot }{\otimes}

\newcommand{\Ext}{{\mathrm{Ext}}}

\newcommand{\Hom }{{\mathrm{Hom}}}
\newcommand{\sHom}{{\mathcal{H}om}}
\newcommand{\tr }{{\mathrm{tr}}}

\newcommand{\Spec}{{\mathrm{Spec}}}

\newcommand{\cA}{{\mathcal{A}}}
\newcommand{\cO}{{\mathcal{O}}}
\newcommand{\cM}{{\mathcal{M}}}

\newcommand{\cC}{{\mathcal{C}}}
\newcommand{\cY}{{\mathcal{Y}}}

\newcommand{\cX}{{\mathcal{X}}}
\newcommand{\cZ}{{\mathcal{Z}}}

\newcommand{\HH}{{\mathbb H}}

\newcommand{\PP }{{\mathbb P}}
\newcommand{\GG }{{\mathbb G}}

\newcommand{\ZZ }{{\mathbb Z}}
\newcommand{\RR }{{\mathbb R}}

\newcommand{\ka }{{\alpha}}

\newcommand{\ch}{\mathrm{ch}}

\newcommand{\lan}{\langle}
\newcommand{\ran}{\rangle}

\newcommand{\LL}{\mathbb{L}}

\newcommand{\td}{\mathrm{td}}
\newcommand{\Ch}{\mathrm{Ch}}
\newcommand{\cB}{\mathcal{B}}

\newcommand{\End}{\mathrm{End}}
\newcommand{\MF}{\mathrm{MF}}

\newcommand{\Perf}{\mathrm{Perf}}
\newcommand{\Mod}{\mathrm{Mod}}

\newcommand{\MC}{\mathrm{MC}}

\newcommand{\vC}{\text{\rm \v{C}}}

\newcommand{\fU}{\mathfrak{U}}

\newcommand{\MFdg}{\mathrm{MF}_{dg}}

\newcommand{\tch}{\ch_{tw}}
\newcommand{\at}{\mathrm{at}}
\newcommand{\can}{\mathfrak{can}}
\newcommand{\Tr}{\mathrm{Tr}}

\newcommand{\bp}{\mathfrak{p}}
\newcommand{\oC}{\overline{C}}
 \newcommand{\rC}{\mathrm{C}}
 \newcommand{\ftw}{\mathfrak{tw}}
  \newcommand{\oa}{\overline{a}}     
  \newcommand{\oMCII}{\overline{\MC}^{II}}

\newcommand{\fc}{\mathfrak{c}}
\newcommand{\Hochp}{C'}
\newcommand{\MCp}{\MC '}

\newcommand{\otau}{\overline{\tau}}
\newcommand{\oMC}{\overline{\MC}}
\newcommand{\cD}{\mathcal{D}}

\newcommand{\dC}{d_{\scriptscriptstyle \text{\v{C}ech}}}

\newcommand{\Kos}{\mathrm{Kos}}

\newcommand{\Coh}{\mathrm{Coh}}

\begin{document}
\title[Riemann-Roch for stacky matrix factorizations]{
Riemann-Roch for stacky matrix factorizations}

\begin{abstract}
We establish a Hirzebruch-Riemann-Roch type theorem and  Grothendieck-Riemann-Roch type theorem for matrix factorizations on quotient Deligne-Mumford  stacks.
For this we first construct a Hochschild-Kostant-Rosenberg type isomorphism explicit enough to yield 
a categorical Chern character formula. We next find an expression of the canonical pairing of Shklyarov
under the isomorphism. 
\end{abstract}  

\author[D. Choa]{Dongwook Choa}
\address{Korea Institute for Advanced Study\\
85 Hoegiro, Dongdaemun-gu \\
Seoul 02455\\
Republic of Korea}
\email{dwchoa@kias.re.kr}

\author[B. Kim]{Bumsig Kim$^*$}
\address{Korea Institute for Advanced Study\\
85 Hoegiro, Dongdaemun-gu \\
Seoul 02455\\
Republic of Korea}

\author[B. Sreedhar]{Bhamidi Sreedhar}
\address{Center for Geometry and Physics\\ Institute for Basic Science (IBS)\\ Pohang 37673\\ Republic of Korea}
\email{sreedhar@ibs.re.kr}
\date{\today}

\thanks{$*$: The second author, Professor Bumsig Kim (1968-2021), sadly passed away during the final stages of the preparation of  this manuscript. 
The first and third authors were both postdocs under his mentorship at KIAS when this project started. We thank Prof. Bumsig Kim  
for his mathematical generosity and kindness. We shall eternally be grateful for the opportunity to have known him not only as a great mathematician but also as a wonderful human being.}



\maketitle

\tableofcontents

\section{Introduction}

\subsection{Main results}
Let $k$ be an algebraically closed field of characteristic zero. The main interest of this paper is an {\em LG model}, $(\cX, w)$, where $\cX$ is a smooth separated Deligne-Mumford stack of finite type over $k$ and a regular function $w$ with no other critical values but zero. 

By a matrix factorization for $(\cX, w)$ we mean a pair $(P, \delta _P)$ of a locally free coherent $\GG$-graded sheaf $P$ on $\cX$ and a curved differential $\delta _P$ whose 
square is $w\cdot \mathrm{id}_P$. Here $\GG$ can be either the group $\ZZ$ or $\ZZ/2$ depending on $w$.
There is the notion of the coderived category of matrix factorizations $\mathrm{D}\MF(\cX, w)$  and its dg enhancement defined as
 the dg quotient of the dg category of matrix factorizations by the subcategory 
 of coacyclic or equivalently locally contractible matrix factorizations. Later we will introduce its dg-enhancement denoted by $\MFdg(\cX, w)$; see \cite{EP, PV: Sing} and also Definition~\ref{def: Cech model}.

\subsubsection{HKR and a Chern character formula} Associated to the dg category $\MFdg (\cX, w)$ there is  the so-called mixed Hochschild chain complex 
\[ \MC(\MFdg (\cX, w)).\]
It has been expected that  $\MC(\MFdg (\cX, w))$ should be quasi-isomorphic to the $dw$-twisted de Rham mixed complex $(\Omega ^{\bullet}_{I\cX}, -dw |_{I\cX} , d)$ 
of the inertia stack $I\cX$ of $\cX$. However only particular cases have been proven so far. In this paper we verify that the expectation is indeed true. 

We first introduce some notation. Let $\rho _{\cX} : I\cX \to \cX$ be the natural forgetful morphism. Write 
$P|_{I\cX}$ for $\rho ^* _{\cX} P$ and let  $\can _{P|_{I\cX}}$ be the canonical automorphism of $P|_{I\cX}$; see \S~\ref{sec: canonical auto}.
Let $\tr$ denote the supertrace morphism 
\[ \RR \Hom (P|_{I\cX}, P|_{I\cX} \ot (\Omega ^{\bullet}_{I\cX}, -dw|_{I\cX} )) \to \HH ^*(I\cX, ( \Omega ^{\bullet}_{I\cX}, -dw|_{I\cX} )) \]
and let  $\hat{\at} (P|_{I\cX}) $ denote  the Atiyah class of the matrix factorization $P|_{I\cX}$ for $(I\cX, w|_{I\cX})$; see \cite{FK, KP, KP2} and \S~\ref{sub: Atiyah}.

\begin{Thm}\label{thm: HKR and Ch}
Suppose $\cX$ is smooth and has the resolution property. Then there is an isomorphism 
\[ \MC (\MFdg (\cX, w)) \cong \RR\Gamma (\Omega ^{\bullet}_{I\cX}, -dw|_{I\cX} , d)  \]
in the derived category of mixed complexes.
Under the isomorphism the Hochschild homology valued  Chern character $\ch _{HH} (P)$ is 
representable by 
\[ \tr \big( \can _{P|_{I\cX}} \exp ( \hat{\at} (P |_{I\cX} )) \big)  \] 
in $\HH ^* ( I\cX, (\Omega ^{\bullet}_{I\cX}, -dw|_{I\cX} ) )$ after the appropriate sense  of the exponential  operation $\exp$ is taken into account; see \S~\ref{sub: Atiyah}. 
\end{Thm}

The  history of related works are very rich. Here we mention  only  the case of  stacky matrix factorizations.
In the local case   Theorem \ref{thm: HKR and Ch} is proven by Polishchuk and Vaintrob \cite{PV: HRR}.
There are works of C\u ald\u arau, Tu and Segal \cite{CT, Se} for HKR type isomorphisms in affine cases with a finite group action. The paper \cite{BFK: Kernels} of
Ballard, Favero, and Katzarkov shows a HKR type isomorphism for the graded cases on linear spaces.  
 This result has also been obtained by Halpern-Leistner and Pomerleano \cite[Remark 3.20]{H-LP} and \cite[Corollary 4.6]{H-LP-Arxiv}. We note that there is a difference in the map constructed in the current text and \cite{H-LP-Arxiv} (see \eqref{eqn: tau} and Remark \ref{rem:commutative} ).  Theorem \ref{thm: HKR and Ch}  is proven by  Kuerak Chung, Taejung Kim, and the second author  in \cite{CKK1}, when one considers quotient stacks of the form $[X/G]$ where $X$ is a smooth variety with a finite group action. 
 In \cite{KP2} a HKR type isomorphism and Chern character formula including the case for the graded matrix factorizations are obtained by the universal Atiyah class.

\subsubsection{HRR and GRR} 
Further assume that the smooth separated DM stack $\cX$ is a stack quotient of a smooth variety by an action of an affine algebraic group and
the critical locus of $w$ is proper over $k$. When $\GG = \ZZ/2$, we shall further assume that the morphism $w: \cX \to \mathbb{A}^1_k$ is flat.
We call the pair $(\cX, w)$ a proper LG model.
We define the Euler characteristic $\chi (P, Q)$ of the pair $(P, Q)$ by the alternating sum of the dimensions of  higher sheaf cohomology:
 \[ \chi (P, Q) := \sum _{i \in \GG}  (-1)^i \dim \RR ^i \Hom (P, Q). \]
For a vector bundle $E$ on $I\cX$ let $\at (E)\in \Ext ^1 (E, E \ot \Omega ^1_{I\cX})$ denote the usual Atiyah class of $E$.
Let   \[ \ch _{tw} (E) := \tr (\can _E \exp (\at (E)) \in \HH ^* ( I\cX, (\Omega ^{\bullet}_{I\cX}, 0 ) )  . \]
For a virtual vector bundle $E$ define $\ch_{tw}(E)$ by linearity. 
We define the Todd class $\td  (T_{I\cX})$ of $T_{I\cX}$ by the formulation of Todd class in terms of the Chern character $\ch _{tw} (T_{I\cX})$; see \S~\ref{sub: Atiyah}.

\begin{Thm} \label{thm: HRR}
Let $P^{\vee}$  denote   the  matrix factorization $(P^{\vee}, \delta _P ^{\vee})$  for $(\cX, -w)$ dual to $(P, \delta _P)$,
let $ N_{I\cX / \cX}$ denote the normal vector bundle of $I\cX$ to $\cX$ via  $\rho _{\cX}$, let $\dim _{I\cX}$ be the locally constant function for local dimensions of $I\cX$,
and let $\lambda _{-1} ( N_{I\cX/\cX}^{\vee})$ be the alternating sum of exterior powers of $N_{I\cX/\cX}^{\vee}$. Then 
\begin{equation} \chi (P, Q) =
\int _{I\cX} (-1)^{\binom{\dim _{I\cX}+1}{2}}  \ch _{HH} (Q) \wedge \ch _{HH} (P ^{\vee})  \wedge \frac{\td  (T_{I\cX})}{\tch (\lambda _{-1} ( N^{\vee}_{I\cX / \cX}))}. \end{equation}
Here RHS is the composition of the following two operations:
\begin{align*}  & \begin{multlined}[b] 
\wedge  : \   \HH ^{*} (I\cX, (\Omega ^{\bullet} _{I\cX}, -dw|_{I\cX} )) \ot \HH ^{*} (I\cX, (\Omega ^{\bullet} _{I\cX}, dw|_{I\cX} )) 
\ot  \HH^{*}(I\cX, (\Omega ^{\bullet} _{I\cX}, 0))  \\  \to  \oplus _p H^{*} _c(I\cX, \Omega ^p _{I\cX}[p]) ; \end{multlined} \\
&  \int_{I\cX} : \ \oplus _{p\in \ZZ} H^{*} _c(I\cX, \Omega ^p _{I\cX}[p]) \xrightarrow{projection}  H^0_c (I\cX, \Omega ^n _{I\cX} [n]) \xrightarrow{\tr _{I\cX}} k .
\end{align*} 
\end{Thm}

%\medskip

Let $(\cY, v)$ be another proper LG model. Consider a proper morphism $f: \cX \to \cY$ with $f^* v = w$.
Let $K_0 (\cA)$ be the Grothendieck group of the homotopy category of a pretriangulated dg category $\cA$.
Let $f_! : K_0 (\MFdg (\cX, w) ) \to K_0 (\MFdg (\cY , v) ) $ be the homomorphism induced by $\RR f_*$ on the Grothendieck groups 
; see \cite[\S~2]{CFGKS}. 
Let $\widetilde{\td} (T_{If}) : = \widetilde{\td} (T_{I\cX })  / If^* \widetilde{\td} (T_{I\cY})$ where
\[  \widetilde{\td} (T_{I\cX}) = \frac{\td (T_{I\cX})}{\tch (\lambda _{-1} ( N_{I\cX/\cX}^{\vee}))} .\]
Let $\dim_{If}$ be the function on $I\cX$ for relative local dimensions of $If : I\cX \to I\cY$
and let 
\[ \int _{If} : \HH ^* (I\cX, (\Omega ^{\bullet}_{I\cX} , - dw)) \to \HH ^* (I\cY , (\Omega ^{\bullet}_{I\cY} , - dv)) \] be the pushforward; see \S~\ref{subsub: basic prop}.

\begin{Thm}\label{Thm: GRR} {\em (=Theorem~\ref{Thm: GRR body})}
The following diagram is commutative:
\begin{equation*} \xymatrix{ K_0 (\MFdg (\cX, w) ) \ar[d]_{\ch_{HH}} \ar[rrr]^{f_{!}} & & & K_0 (\MFdg (\cY , v )) \ar[d]^{\ch_{HH}} \\ 
                                            \HH ^{0} (I\cX, (\Omega ^{\bullet}_{I\cX}, -dw|_{I\cX})) \ar[rrr]_{  \int _{If}  (-1)^{\dim_{If}} \cdot \wedge \widetilde{\td} (T_{If}) } & & 
                                            &  \HH ^{0} (I\cY, (\Omega ^{\bullet}_{I\cY}, -dv|_{I\cY})) . } \end{equation*}
\end{Thm}

 When $\GG = \ZZ$ (and hence $w=0$), various versions of Riemann-Roch theorem on DM stacks are proven by
Kawasaki \cite{Kaw}, To\"en \cite{Toen}; and Edidin and Graham \cite{EG1, EG2, EG3}.
When $w$ has one critical point, Theorem \ref{thm: HRR} is proven by Polishchuk and Vaintrob \cite{PV: HRR}.

\subsection{On the proofs and pertinent works}

\subsubsection{HKR} For the computation of Hochschild homology of the category of matrix factorizations, 
 there are at least  three known approaches by
(1) finding a suitable flat resolution of the diagonal module \cite{BFK: Kernels, LP, PV: HRR}, (2) using the quasi-Morita equivalence \cite{BW, CT, CKK1, Ef, Se}, and
(3) using the universal Atiyah classes \cite{KP, KP2} which goes back to \cite{Cal, Ma}. 
In this paper we take the second approach  by achieving a correct globalization of Baranovsky's map \cite{Bar}  closely following  the proof of Proposition 2.13 of \cite{H-LP} and \cite[Corollary 4.6]{H-LP-Arxiv}.
Combining this with a chain-level map from \cite{BW, CKK1} we obtain a boundary-bulk map formula as well as a Chern character formula; see \S~\ref{sec: Chern char form}.

\subsubsection{HRR}
For any proper smooth dg category $\cA$ there is a categorical HRR theorem by Shklyarov \cite{Shk: HRR}. Let $\cA ^{op}$ denote the opposite category
of $\cA$.
Let  $\lan , \ran _{can}$ be the canonical pairing of Shklyarov:
\[ \lan , \ran _{can} : HH_* (\cA) \ot HH_* (\cA ^{op}) \to k . \]
Then the categorical HRR theorem is the equality 
\[ \chi (P, Q) = \lan \Ch_{HH} (Q),  \Ch_{HH} (P^{\vee}) \ran _{can} \ \forall P, Q \in \cA.\]

There is a characterization property of the canonical pairing in terms of the Chern character of diagonal bimodule; see for example \S~\ref{sub: char prop}. 
Let $\cA$ be the dg category of matrix factorizations for $(\cX, w)$
localized by coacyclic matrix factorizations.
When $\cX$ is local, using the characterization property 
Polishchuk and Vaintrob \cite{PV: HRR} show
that the canonical pairing becomes up to sign the residue pairing under their HKR type isomorphism. 
In the nonstacky local case, there is also a work of Brown and Walker \cite{BW: conj} identifying the canonical pairing with the residue pairing
under the HKR type isomorphism.  

When $\cX$ is a smooth variety, using the deformation to the normal cone as well as the characterization property of 
the canonical pairing, the second author \cite{Kim} shows that
 the canonical pairing becomes a trace map under the HKR type isomorphism up to a Todd correction term. 
 When $\cX$ is stacky, furthermore using the deformation to the normal cone for {\em local immersions} \cite{Kresch: cycle, Vistoli}
 and the Chern character formula in Theorem~\ref{thm: HKR and Ch},
  we are able to prove Theorem~\ref{thm: HRR}.

\subsubsection{GRR}
 The proper morphism $f$ in Theorem \ref{Thm: GRR} induces a dg functor from $\MFdg (\cX, w)$ to $\MFdg (\cY , v)$.
 The induced homomorphism \[ HH_* (\MFdg (\cX, w)) \to HH_* (\MFdg (\cY , v)) \] in Hochschild homology has a description \S~\ref{sub: cat grr}
 in terms of the canonical pairing of  $\MFdg (\cX \ti \cY , - w \ot 1 + 1 \ot v)$ and the categorical Chern
 character of the matrix factorization associated to the graph morphism $\Gamma _f : \cX \to \cX \ti \cY$. Under the HKR isomorphisms, the deformation to the normal cone allows us
 to interpret the description as the pushforward by $f$ of twisted Hodge cohomology  up to a Todd correction term.

\subsection{Conventions and notation}
Let the ground field $k$ be an algebraically closed field of characteristic zero.
Let $\mu _r$ denote the group of $r$-th roots of unity over the field $k$.
Throughout this paper, let $\cX$ be a finite type separated DM stack  over $k$.
We denote by $I\cX$ the inertia stack of $\cX$. 
Let \[ \rho _{\cX} : I\cX \to \cX \] denote the natural representable morphism, which is finite and unramified;  \cite[\S~3]{AGV}.
Let $G$ be a finite group which acts on a scheme $Y$ of finite type over $k$. 
For a quotient stack $[Y/G]$, let 
\[ IY := \{ (g, y) \in G \ti Y : gy = y \} := G\ti Y \ti _{Y\ti Y} \Delta _Y \] so that $I[Y/G] = [IY/G]$.

We write $\widehat{G}$ for the character group $\mathrm{Hom}(G, \mathbb{G}_m)$ of $G$. 
The label $\sim$ on an arrow indicates the arrow is a  quasi-isomorphism.
By a vector bundle we mean a locally free coherent sheaf.

For a local immersion (i.e., an unramified representable morphism) $f: \cX \to \cY$ between DM stacks, we denote  $C_{\cX / \cY}$ be 
the normal cone to $\cX$ in $\cY$; see \cite{Kresch: cycle, Vistoli}. If $f$ is a regular local immersion, then
we write $N_f$ or $N_{\cX/ \cY}$ for the vector bundle  $C_{\cX / \cY}$ on $\cX$.

For a vector bundle $E$, we often write $1_E$ for the identity morphism $\mathrm{id}_E$ of $E$. For a dg category $\cA$, its homotopy category is denoted by $[\cA]$.

\subsubsection*{\bf Acknowledgements}
We thank David Favero and Daniel Pomerleano for answering some questions that we had during the preparation of this manuscript. We thank Jeongseok Oh for comments on an initial draft. The third author thanks Charanya Ravi for discussions.

B. Kim  is supported by KIAS individual grant MG016404 at Korea Institute for Advanced Study. 
D. Choa is supported by KIAS individual grant MG079401  at Korea Institute for Advanced Study.
B. Sreedhar  was supported by the Institute for Basic Science (IBS-R003-D1).

\section{Mixed Hochschild complexes and Chern characters}

Unless otherwise stated, we follow notation and conventions of \cite{BW, CKK1} for curved dg (in short cdg) categories $\cA$,
the mixed Hochschild complexes of $\cA$, and the category of mixed complexes. We briefly recall notation therein 
and some foundational facts which we are going to use later.

 \subsection{Mixed Hochschild complexes}

  For a curved dg category $\cA$, we {use the following notation:
 
 \begin{itemize}
 \item $C (\cA) \ (\MC (\cA))$ : (mixed) Hochschild complex.
 \item   $\oC (\cA)\ (\oMC (\cA))$:   (mixed) normalized Hochschild complex.
 \item  $C ^{II} (\cA) \ (\MC^{II} (\cA))$ : (mixed) Hochschild complex of the second kind.
 \item $\oC ^{II} (\cA)\  (\oMC^{II} (\cA))$ : (mixed) normalized Hochschild complex of the second kind. 
 \end{itemize}
  For notational convenience we let $C'$ denote either $C, \oC, C^{II}$ or $\oC ^{II}$ and  $\MC ' $ denote either $\MC, \oMC$, $\MC ^{II}$, or $\oMC ^{II}$.
 The normalized negative cyclic complex is denoted by $\oC (\cA) [[u]]$ where $u$ is a formal variable of degree $2$.
 For a mixed complex $(C, b, B)$ we simply write $C[[u]]$ for the complex $(C [[u]], b + uB)$.
 
\subsection{Foundational facts frequently in use}

\subsubsection{Invariance of the natural projections} The projection $\MC (\cD) \to \oMC (\cD) $ for a dg category $\cD$ and the projection 
$\MC ^{II} (\cA) \to \oMC ^{II} (\cA )$ for a cdg category $\cA$ are quasi-isomorphisms \cite{BW, PP}. 

\subsubsection{(Quasi-)Morita invariance}\label{sub: quasi-Morita} 
If a dg functor  $\cD \to \cD'$ is Mortia-equivalent (i.e., the induced functor $D(\cD) \to D (\cD')$ of derived categories of $\cD$ and $\cD'$
is an equivalence), then the induced morphism $\MC (\cD ) \to \MC (\cD ')$ of mixed complexes is a quasi-isomorphism \cite{Keller_Invariance}.   
If $\cA \to \cA'$ is a pseudo-equivalence of cdg categories, then the induced morphism 
$\oMC ^{II} (\cA) \to \oMC ^{II} (\cA ')$ is a quasi-isomorphism \cite{PP}. This invariance is dubbed as quasi-Morita invariance. 

\subsubsection{Localization in cyclic homology}\label{subsub: localization} 
If $\cA \to \cB \to \cC$ be an exact sequence of exact dg categories, then it induces 
an exact triangle of  mixed complexes 
\[ MC (\cA) \to \MC (\cB) \to \MC (\cC) \to \MC (\cA ) [1]; \]
see \cite[\S~5.6]{Keller_Cyclic}. 

\subsubsection{Local description of the inertia stack} Let $G$ be a finite group acting on a $k$-scheme $X$ and let $G/G$ denote the set of conjugacy classes of $G$.
 Then the following holds
$I[X/G] \cong  \sqcup _{g \in G/G}  [X^g /\rC _{G}(g) ]  $.

\subsubsection{Invariants and coinvariants}  Let $G$ be a group with a linear action upon $V$ a not-necessarily finite dimensional vector space.
Define the invariant space $V^G:=\Hom _G (k, V)$ and the coinvariant space  $V_G := k \ot _G V$.
If $G$ is a finite group, then the composition of the natural homomorphisms $V^G \to V \to V_G $ is an isomorphism.

\subsection{Categorical Chern characters}
For $P \in \cA$ the cycle class represented by  the identity morphism $1_P$ in the normalized Hochschild complex $\oC (\cA)$ (resp. the normalized negative cyclic complex
$\oC (\cA) [[u ]]$) 
of $\cA$  is denoted by $\Ch _{HH} (P)$ (resp. $\Ch _{HN}$).

\newcommand{\Aut}{\cA ut}

\section{The canonical central automorphism} \label{sec: canonical auto}
\subsection{The central embedding}\label{sub: central embedding} The inertia stack has a decomposition: \[ I\cX = \sqcup _{r=1}^{\infty} I _{\mu_r} \cX . \] 
An object of $I_{\mu_r}\cX$ over a $k$-scheme $T$ is a pair $(\xi , \ka)$ of an object
$\xi \in \cX (T)$ and  an injective morphism of group-schemes $\ka : \mu _r \ti T \to \Aut _T (\xi )$; see \cite[\S~3]{AGV}.
Note that for all but finitely many $r$, $I_{\mu_r} \cX$ is empty.
An automorphism of the pair $(\xi, \ka)$ in $I_{\mu_r}\cX$ is by definition an automorphism $f \in \Aut _T (\xi )$ such that 
\[ f \circ \ka \circ f^{-1} = \ka . \] In other words the automorphism group-scheme $\Aut _T (\xi, \ka )$ of $(\xi, \ka)$ over $T$ is the centralizer of $\ka$ in  $\Aut _T (\xi )$.
We have a canonical central embedding $\fc: \mu _r \ti T \to \Aut _T(\xi, \ka)$. 
This gives a natural morphism \begin{equation*}\label{eqn: c}  \mu _r \ti T \to T\ti _{I_{\mu_r}\cX} T ; (\xi, t) \mapsto (t, t, \fc (\xi , t)) .\end{equation*}

\subsection{The central automorphism}
Let $T \to I_{\mu_r}\cX$ be an \'etale surjection and let $pr _i :  T\ti _{I_{\mu_r}\cX} T  \to T$ be the $i$-th projection.
A vector bundle $E$ on $I_{\mu_r}\cX$ amounts to a vector bundle $F$ on $T$ with an isomorphism $\phi ^F \in \mathrm{Isom}_T (pr _1^* F , pr_2^* F)$ satisfying the cocycle condition.
By pulling back the isomorphism $\phi ^F$ to $\mu _r \ti T$ we obtain a morphism of group-schemes  
$\mu _r\ti T \to \mathrm{Aut}_{T} (F)$. Here $\mathrm{Aut}_{T} (F)$ denotes the group of automorphisms of $F$ fixing $T$.
Since $\fc$ is central, the homomorphism descends to a homomorphism $\mu _r \to \mathrm{Aut}_{I_{\mu_r}\cX} (E)$.
Denote  by \[ \can _E \in  \mathrm{Aut}_{I_{\mu_r}\cX} (E) \] the image of the
chosen $r$-th root $e^{2\pi i/r}$ of unity.

According to the action $\mu_r$ upon $E$, the bundle $E$  is decomposable into eigenbundles  \[ \oplus _{\chi \in \widehat{\mu _r}} E_{\chi} ,\] 
where  $\widehat{\mu _r}$ is the character group of $\mu_r$.
Then we have \[ \can _E = \oplus_{\chi \in \widehat{\mu _r} } \chi (e^{2\pi i /r } ) \mathrm{id}_{E_{\chi}} \in \Hom _{I_{\mu_r}\cX} (E, E) .\]
In turn this gives an automorphism  $\can _P$  of $P$ for every object $P$ of $\MFdg (I\cX, w)$. It is a morphism in the category $\MFdg (I\cX, w)$.

Note  that \begin{equation}\label{eqn: center} \can _{P'} \circ a = a \circ \can _{P} \quad \forall a \in \Hom _{I\cX}(P, P'). \end{equation} 
Hence the assignments $P\mapsto \can _P$ yield 
a natural transformation $\can$ between the identity functor $\mathrm{id} :  \MFdg (I\cX, w) \to  \MFdg (I\cX, w)$.
We will often drop the subscript $P$ in $\can _P$ for simplicity. 

\subsection{The local description}
Locally the automorphisms are described as follows. Suppose that $\cX = [X/G]$ where $G$ is a finite group and $X$ is a scheme. 
Let $g\in G$ with order $r$ and write $\rC _{G}(g)$ for the centralizer of $g$ in $G$.  For the component $[X^g / \rC _{G}(g) ]$ of $I_{\mu_r}\cX$ and a $\rC _{G}(g)$-equivariant sheaf $E$ on $X^g$,
we have an isomorphism \[ (g^{-1})^* E \xrightarrow{\varphi _g^E} E \] from the equivariant structure of $E$. Since $g$ acts trivially on $X^g$, $(g^{-1})^* E = E$. And hence
$\varphi _g^E $ is an automorphism of $E$, which is the automorphism $\can _E$.
Since any element of $\rC _{G}(g)$ commutes with $g$, the homomorphism $\varphi ^E_g$ is  $\rC _{G}(g)$-equivariant. 
Thus  $\varphi_g^E \in \mathrm{Hom}_{I_{\mu_r}\cX} (E, E)$. For any  $\rC _{G}(g)$-equivariant sheaf $E'$  on $X^g$ and any $\rC _{G}(g)$-equivariant 
$\cO_{X^g}$-module homomorphism $a: E\to E'$, note that
$\varphi _g^{E'}  \circ a = a \circ \varphi _g^E $.

\subsection{$T_{I\cX} \cong IT_{\cX}$}\label{sub: TI IT} In this subsection let $\cX$ be smooth over $k$.
We prove that there is a natural isomorphism $T_{I\cX} \cong IT_{\cX}$.

\subsubsection{}\label{subsub: ses rho TX}    
Note first that there is a natural short exact sequence
of vector bundles   \[ 0 \to  T_{I\cX} \to \rho ^* _{\cX} T_{\cX} \to N_{\rho_{\cX}} \to 0 . \]
 In fact this sequence splits. The reason is that according to the canonical automorphism of $\rho ^* _{\cX} T_{\cX}$
there is a decomposition of $\rho _{\cX} ^* T_{\cX} $ into the fixed part  and the moving part,
 which are naturally isomorphic to  $T_{I\cX} $  and   $N_{\rho_{\cX}}$, respectively.

\subsubsection{}\label{subsub: TI IT}
Consider a commuting diagram of natural morphisms
\[ \xymatrix{  I T_{\cX} \ar@{-->}[rd]^{\phi}    \ar@/^1pc/[rrd] \ar@/_1pc/[ddr]       &                                                       &                        \\
                                                                      &     T_{\cX}|_{I\cX}  \ar[r] \ar[d]        &  T_{\cX} \ar[d]  \\
                                                                      &           I\cX \ar[r]                              & \cX , }              \] 
                                                                      where the square is a fiber square.
 \begin{lemma}\label{lem: IT TI} The morphism $\phi$ induces an isomorphism $IT_{\cX}  \cong T_{I\cX}$.                                            
 \end{lemma}
  \begin{proof} First note that it is enough to check the isomorphism over the \'etale site of the coarse moduli space of $\cX$. 
  Since $\cX$ is separated, then $\cX$ is \'etale locally a quotient of a nonsingular variety $Y$  by a finite group $G$ action.
 Hence we may assume that $\cX = [Y/G]$. 
 Since \[ T_{[Y/G]} \cong [T_Y / G]  \text{ and }    I[Y/G] \cong \bigsqcup _{g\in G/G} [Y^g / \rC _{G}(g) ],  \]
we have
\[  IT_{[Y/G]} \cong \bigsqcup _{g \in G/G} [ (T_Y)^g / \rC _{G}(g) ]   \text{ and } T_{I[Y/G]} \cong \bigsqcup _{g\in G/G} [T_{Y^g} /\rC _{G}(g) ]. \]
  Since $(T_Y)^g \cong T_{Y^g}$, we conclude the proof.            
\end{proof}

\begin{Rmk}
When $\cX$ is the global quotient $[Y/G]$ by a finite group $G$, $N_{\rho _{\cX}} \cong [N_{Y^g/ Y} / \rC _{G}(g) ]$.
\end{Rmk}

\section{Stacky Hochschild-Kostant-Rosenberg}
In this section, we assume $\cX$ is a smooth separated DM stack and the critical value of $w$ is over $0$.  
If $\GG = \ZZ$, then $w=0$. 
If $\GG = \ZZ/2$, then we furthermore assume that $w: \cX \to \mathbb{A}^1$ is flat.

\subsection{Matrix factorizations and their derived categories}
Whenever an LG model $(\cX , w)$ is given, we consider a sheaf of curved differential graded (CDG for short) algebra $(\cO_{\cX} , -w)$ over $\cX$. It is concentrated in degree $0$ with zero differential and a curvature $-w$.
\begin{Def}\label{Def: CDG module}
A {\em quasi-differential graded module (QDG-module for short)} over $(\cO _{\cX} , -w)$ is a pair $(P, \delta _P)$ of an $\cO _{\cX}$-module $P$ and an  $\cO _{\cX}$-linear degree $1$ endomorphism $\delta _P$.
We say a QDG module is
    \begin{itemize}
    \item {\em (quasi-)coherent } if $P$ is (quasi-)coherent,
    \item {\em locally free} if $P$ is locally free, 
    \item {\em matrix quasi-factorization} if $P$ is locally free of finite rank.
    \end{itemize}
\end{Def}

We denote a category of QDG modules over $(\cO_\cX, -w)$ by $q\mathrm{Mod}(\cX, w)$. It is a CDG category whose morphisms and differentials are
    \begin{align*}
    \Hom_{q\mathrm{Mod}}&((P,\delta_P), (Q, \delta_Q))=\left( \Hom_{\cO_\cX}(P, Q) , \delta \right),\\
    \delta(f)&=  \delta _Q \circ f  - (-1)^{|f|} f \circ \delta _P.
    \end{align*} 
    The curvature element $h_{(p, \delta_P)}$ of $(P, \delta_P)$ is defined as $\delta_P^2+\rho_{-w} \in \End(P)$, where $\rho _{-w}$ is the multiplication map by $-w$.

    \begin{Def}
    A QDG-module $(P, \delta_P)$ is called {\em a factorization} if its curvature is zero. 
    We define (quasi-)coherent or locally free factorizations as in \ref{Def: CDG module}. 
    In particular, we call it a {\em a matrix factorization} if $P$ is a locally free of finite rank.
    \end{Def}
    
By definition, factorizations form a DG subcategory inside $q\mathrm{Mod}(\cX, w)$ denoted by $\mathrm{Mod}(\cX, w)$. We denote a full DG subcategory of (quasi-) coherent and matrix factorizations by $\mathrm{QCoh}(\cX, w)$,  $\Coh(\cX, w)$ and $\mathrm{MF}(\cX, w)$ respectively. 

We recall constructions of the derived category of factorizations following \cite{Pos: two}. Let $[\mathrm{QCoh}(\cX, w)]$ be the homotopy category of $\mathrm{QCoh}(\cX, w)$. Denote by $\mathrm{AbsAcyc}(\cX, w)$ the smallest triangulated subcategory containing the totalizations of all short exact sequences in $Z^0\mathrm{QCoh}(\cX, w)$. Its object are called {\em  absolutely acyclic factorizations.} Also, denote by $\mathrm{CoAcyc}(\cX, w)$ the smallest triangulated subcategory containing the totalizations of all acyclic factorizations which is closed under infinite direct sum, its object are called {\em a coacyclic factorizations}.
    \begin{Def}
    The {\em absolute derived category of $\mathrm{QCoh}(\cX, w)$} is the Verdier quotient
    \[\mathrm{D}^{\mathrm{abs}}(\mathrm{QCoh}(\cX, w)):=[\mathrm{QCoh}(\cX, w)]/\mathrm{AbsAcyc}(\cX, w).\]
    The {\em coderived category of $\mathrm{QCoh}(\cX, w)$} is the Verdier quotient
    \[\mathrm{D}^{\mathrm{co}}(\mathrm{QCoh}(\cX, w)):=[\mathrm{QCoh}(\cX, w)]/\mathrm{CoAcyc}(\cX, w).\]
    We define an absolute/coderived category of  $\Coh(\cX, w)$ and $\mathrm{MF}(\cX, w)$. 
    \end{Def}
    \begin{Def}
    The derived category of of matrix factorizations denoted by $\mathrm{D}\MF(\cX, w)$ is the smallest full triangulated subcategory of $\mathrm{D}^{\mathrm{co}}(\mathrm{QCoh}(\cX, w))$ which contains $\mathrm{D}^{\mathrm{abs}}(\MF(\cX, w))$.
    \end{Def}

    \begin{Rmk}Relations between various categories are well-known. We only recall a few facts we will going to use later. 
    If $\cX$ is smooth, then it is known that Verdier localization 
    $\mathrm{D}^{\mathrm{abs}}(\mathrm{QCoh}(\cX, w)) \to \mathrm{D}^{\mathrm{co}}(\mathrm{QCoh}(\cX, w))$
    is an equivalence, and an image of $\mathrm{D}^{\mathrm{abs}}(\Coh(\cX, w))$ consists of compact generators inside $\mathrm{D}^{\mathrm{co}}(\mathrm{QCoh}(\cX, w))$ (see \cite[\S~3.6]{Pos: two}). If $\cX$ has the resolution property, then 
    $\mathrm{D}^{\mathrm{abs}}(\MF(\cX, w)) \to \mathrm{D}^{\mathrm{abs}}(\mathrm{Coh}(\cX, w))$
    is an equivalence. (See \cite{PV: Sing})
    \end{Rmk}

\subsection{\v{C}ech model}\label{sub: Cech model}
In this subsection, we recall the \v{C}ech type DG enhancement of $\mathrm{D}\MF(X, w)$ as in \cite{CKK1}.

Fix an affine \'etale surjective morphism $\bp : \fU \to \cX$ from a $k$-scheme $\fU$. 
Let  $\fU ^r$ denote the $r$-th fold product of $\fU$ over $\cX$ and let $\bp_r : \fU ^r \to \cX$ denote  the projection. 
For a vector bundle $E$ on $\cX$, let $\vC ^r  (E)= \bp_{r*} \bp_{r}^{*} E$  and 
\[\vC (E) := \left( \bigoplus _{r \ge 1} \vC ^r  (E), \dC \right) = \left[ 0 \to p_{1*}p_1^*E \to p_{2*}p_2^*E \to \cdots \right]\]
a \v{C}ech complex.

Now let $(E, \delta_E)$ be a matrix quasi-factorization over $(\cO_\cX, -w)$. Observe that $\vC(\cO_{\cX})$ can be viewed as a sheaf of $\cO_\cX$-algebras equipped with Alexander-Whitney product. 
By projection formula $ \vC(E) = E \ot _{\cO_{\cX}}\vC(\cO_{\cX})$ and $\vC(E)$ carries a natural $\vC(\cO_\cX)$-module structure. 
We equip $\vC (E)$ with a curved differential 
\[   \delta _{\vC(E)}:= \delta _P \ot 1 + 1\ot \dC  . \]
We regard $(\vC(E), \delta_{\vC(E)})$ as a QDG-module over $(\vC(\cO_\cX), w)$. Notice that the curvature of $(\vC (E), \delta_{\vC(E})$ coincides with the curvature of $E$ under the canonical map
$\End _{\cO _{\cX}} (E) \to \End _{\vC (\cO _{\cX})} (\vC (E))$.

\begin{Def}\label{def: Cech model}
For a fixed affine \'etale open cover $\fU$, a {\em \v{C}ech model CDG category} $q\MFdg(\cX, w)$ is a CDG category whose objects are matrix quasi-factorizations and its $\GG$-graded Hom spaces are defined as usual
    \begin{align}
    \Hom_{q\MFdg}(P, Q)&:=\left( \Hom_{\cO_{\cX}}(\vC(P), \vC(Q)) , \delta \right)\\
    \delta &=  \delta _{\vC (P')} \circ f  - (-1)^{|f|} f \circ \delta _{\vC(P)}
    \end{align}
Similary, we define a {\em \v{C}ech model dg  category} $ \MFdg  (\cX, w) $ of matrix factorizations for $(\cX, w)$ as a full DG subcategory of $q\MFdg (\cX, w) $ consisting of matrix factorizations for $(\cX, w)$. When it is necessary to specify the covering $\fU$, we write $ \MFdg   (\cX, w; \fU) $ for $\MFdg (\cX, w)$.
\end{Def}

    \begin{Rmk}\label{Rmk: truncation}
    We view $\Hom_{\MFdg}(P, Q)$ as a $\mathbb Z \times \mathbb Z/2$ -graded bicomplex. 
    This complex is  not bounded above in $\mathbb Z$-direction because a \v{C}ech cover $\fU$ of a stack is genuinely unordered. 
    One can go around the subtleties by taking a suitable truncation. Suppose $(E, \delta_E)$ is a matrix factorization. Define 
\[ \tau\vC(E) :=  \tau _{ \le \dim \cX} \left( \bigoplus _{r \ge 1} \vC ^r  (E), \dC \right)\]
where $\tau _{r \le \dim \cX}$ denotes the truncation. Note that under the assumptions in this text all stacks have finite cohomological dimension, hence the truncated one does compute the sheaf cohomology of the quasi-coherent sheaf $E$,
since the map to the coarse moduli $\cX \to \underline{\cX}$ is cohomologically affine.
Therefore, the induced map between the spectral sequences associated to \v{C}ech filtrations
\[\Hom(\vC(P), \vC(Q)) \to \Hom(\tau \vC(P), \vC(Q)) \leftarrow \Hom(\tau\vC(P), \tau\vC(Q))\]
are isomorphisms at the first page and hence, are quasi-isomorphisms themselves. Also notice that $\tau \vC(E) \to \vC(E)$ is a quasi-isomoprhism in $\mathrm{D}\MF(\cX, w)$.
\end{Rmk}

It is not hard to see that $\MFdg(\cX, w)$ is a DG enhancement of the derived category of $\mathrm{D}\MF(\cX, w)$. Note that 
    \begin{enumerate}
    \item $P$ is coacyclic if and only if $P|_{\fU}$ is coacyclic, since the natural morphism $P \to \vC (P)$ is termwise exact. See \cite[Proposition 2.2.6]{CFGKS}.
    \item TFAE: $P|_{\fU}$ is coacyclic, $P|_{\fU}$ is absolutely acyclic, and $P|_{\fU}$ is contractible by \cite[\S~3.6]{Pos: two};
    \item TFAE: $P|_{\fU}$ is coacyclic, $P$ is locally contractible by \cite[Proposition 2.2.6]{CFGKS} and (1).
Here $P$ is called {\em locally contractible} if there is an open covering $\mathfrak{V}$ in smooth topology of $\cX$ such that $P|_{\mathfrak{V}}$ is contractible; see \cite{PV: Sing}. 
    \end{enumerate}
Therefore if $P \in \MFdg(\cX, w)$ represents a coacyclic object in $\mathrm{D}\MF(\cX, w)$, then $P$ is locally contractible and hence, $\Hom_{\MFdg}(P, P) \simeq 0$.

\begin{Rmk} \label{rmk: ind cover}
Consider another affine covering $\fU' \to \cX$. Let $\fU '' = \fU ' \ti _{\cX} \fU$. Then there is 
a natural dg functor $\MFdg (\cX, w, \fU') \to \MFdg (\cX, w, \fU '')$, which is a quasi-equivalence. 
The induced chain map  $\MC '  (\MFdg (\cX, w, \fU)) \to \MC '( \MFdg (\cX, w, \fU ''))$ between mixed complexes are quasi-isomorphism. 
For the mixed complex of the first kind it follows from the Morita invariance. 
For that of the either kind consider the Hochschild complex $C '  (\MFdg (\cX, w))$ filtered by \v{C}ech degree.
Since the filtration is bounded (See \ref{Rmk: truncation}), we may apply the Eilenberg-Moore comparison theorem (\cite[Theorem 5.5.11]{Weibel}) 
to check that the induced chain map is a quasi-isomorphism.
A similar argument  shows that the natural chain map  \[ \MC '  (q\MFdg (\cX, w, \fU)) \to \MC '( q\MFdg (\cX, w, \fU '')) \] is also a quasi-isomorphism.
\end{Rmk}

\subsection{Twisting} 
 For a morphism from a scheme $U$  to  $\cX$, write $IU$ for the fiber product $U \ti _{\cX} I\cX$. 
Following To\"en, Halpern-Leistner and Pomerleano \cite{H-LP} we consider the assignments 
\[ U \mapsto  C ' ( q\MFdg (IU, w|_{IU}, \fU \ti _{\cX} IU )) \]  for all \'etale morphisms $U\to \cX$.
They form a presheaf   of mixed complexes on the small \'etale site $\cX$, which we denote by
\begin{align*}  \MC ' (q\MF_{I\cX, w} (-))   . \end{align*}
Denote the associated cochain complex by $C '  (q\MF_{I\cX, w} (-))$. 
There is a natural morphism of mixed complexes \[ nat: \MCp (q\MFdg (I\cX, w)) \to (\Gamma (\cX,  \MCp (q\MF_{I\cX, w} (-))), \]
where the \v{C}ech model $q\MFdg (I\cX, w)$ uses the affine cover $\fU \ti _{\cX} I\cX \to I\cX$.

Define a $k$-linear map $\mathfrak{tw}: C ' (q\MFdg (I\cX, w)) \to C ' (q\MFdg (I\cX, w))$  associated to $\can$ 
by  

\[ a_0 [a_1 | \cdots | a_n ] \mapsto   a_0 [ a_1 | \cdots | \can \circ a_n ] .\]
Note that by \eqref{eqn: center} $b \circ \ftw = \ftw \circ b$, i.e., $\ftw$ is a chain automorphism of the Hochschild complex $\Hochp (q\MFdg (I\cX, w))$.

Consider the composition $\tau ' $ of a sequence of chain maps
\begin{multline}\label{eqn: tau}   \Hochp (q\MFdg (\cX, w)) \xrightarrow{pullback} \Hochp (q\MFdg (I\cX, w)) \\
\xrightarrow{\mathfrak{tw}} \Hochp (q\MFdg (I\cX, w)) 
\xrightarrow{nat} \Gamma (\cX, \Hochp (q\MF_{I\cX, w} (-) )) . \end{multline}

\begin{Prop}\label{prop: local}  The chain map $ \otau ^{II}$ is a quasi-isomorphism when
$\cX$ is of form $[\Spec A / G]$ for some commutative $k$-algebra $A$ with a finite  group $G$ action.
\end{Prop}

A proof of the above proposition will be given \S~\ref{sub: local} and \S~\ref{sub: pf of prop local}.
For simplicity we will often write $\tau$ for $\tau '$ when there is no risk of confusion.

\subsection{Local case}\label{sub: local}

Let $\cX = [\Spec A / G]$ and let $w \in A^G$ a $G$-invariant element of $A$ as in Proposition \ref{prop: local}. 
Let $\MFdg ^G (A, w)$ denote the dg category of $G$-equivariant factorizations $P$ for $(A, w)$ which are projective as $A$-modules.  
The Hom space from $P$ to $Q$ is the $G$-invariant part of $\Hom _A (P, Q)$ of $\GG$-graded $A$-module homomorphisms.
Likewise we have  the cdg category $q\MFdg ^G (A, w)$ of $G$-equivariant quasi-modules for $(A, w)$ which are projective as $A$-modules.     
In fact these coincide with the \v{C}ech models $\MFdg (\cX, w)$ and $q\MFdg (\cX, w)$ with respect to the natural choice of an affine cover: $\Spec A \to \cX$.

 Let $I_g$ be the ideal of $A$ generated by
$a - g a $ for all $a \in A$.
Denote  $A_g := A / I_g$ and $w_g := w |_{A_g} \in A_g$.  
We regard the pair $(A, w)$ (resp. $(A_g, w_g)$) as a curved dg algebra $A$ (resp. $A_g$) with zero differential and curvature $w$ (resp. $w_g$).

The algebra $A$ has the induced left $G$-action. Note that for  $a, b \in A$ and $g, h \in G$,
\[ g (a \  h (b)) = g (a) \  gh (b) .\] 
The cross product algebra $A\rtimes G := A \ot k[G]$ has the multiplication 
defined by $(a\ot g) \cdot (b \ot h) = a g (b) \ot g h$.
We also view $A\rtimes G$ as a right $A$-module with a left $G$-action by
\[ (a \ot g ) \cdot b = a g (b) \ot g \text{ and } h\cdot (a \ot g)  = a \ot g h^{-1} . \]
Equivalently, $A\rtimes G$ is a right $A \rtimes G$-module by the multiplication
\[ (a \ot g) \cdot (b \ot h') = a g(b) \ot g h' . \]
Note that the curvature of $A\rtimes G$ with zero differential as a right quasi-module over $(A \rtimes G, w \ot 1)$
is $-w\ot 1$.

Denote by  $\{ (A\rtimes G, -w\ot 1) \}$ the full subcategory of $q\MFdg (A\rtimes G, w\ot 1)$ consisting 
of the indicated object $A\rtimes G$ with zero differential and curvature $-w\ot 1$ .
The embedding $\{ (A\rtimes G, -w) \} \hookrightarrow  q\MFdg (A\rtimes G, w\ot 1) $ is called a quasi-Yoneda embedding.
It is a pseudo-equivalence; see \cite{PP}.
Consider the embedding $ q\MFdg (A\rtimes G, w\ot 1) \hookrightarrow q\MFdg ^G (A, w)$, which is also a pseudo-equivalence; see \cite{CKK1}. 
Hence by the quasi-Morita invariance  the induced morphism of mixed complexes 
\[ \oMC^{II} (A\rtimes G, -w)  \to  \oMC^{II} ( q\MFdg ^G (A, w) ) \] is a quasi-isomorphism. 

Consider an embedding  $ \{ (A_g \rtimes G, -w_g \ot 1)\}  \hookrightarrow  q\MFdg (A_g, w_g)$. This is a pseudo-equivalence, 
 since $(A_g, -w_g)$ is a direct summand of  $(A_g \rtimes G, -w_g \ot 1)$ as a right quasi-module over $(A_g, -w_g)$.
  Hence  by the quasi-Morita invariance  the induced morphism of mixed complexes 
    \[ \oMC^{II} (\End _{A_g} (A_g \rtimes G) , -w_g \ot 1)  \to  \oMC^{II} (q\MFdg (A_g, w_g))  \]
            is a  quasi-isomorphism.
    
We have a diagram of chain maps
 \begin{equation}\label{eqn: big diag} \mbox{\tiny \xymatrix{ 
\oC ^{II} (q\MFdg ^G (A, w)) \ar[r]^{\tau \ \ \ } & (\oplus _g \oC ^{II} (q\MFdg (A_g, w_g) ) )^G \ar[r]^{natural }_{\cong\ \ } &  (\oplus _g \oC^{II} (q\MFdg (A_g, w_g) ) )_G \\
                    \oC ^{II} (A \rtimes G, - w\ot 1 ) \ar@{^{(}->}[u]^{\sim} \ar[rr]^{ \tau |_{A\rtimes G} } \ar[rd]_{|G|\cdot \psi_{A}}^{\sim} 
                    &  
                    &   \ar@{^{(}->}[u]^{\sim}  \ar[ld]^{\Tr}_{\sim}  (\oplus _g \oC ^{II} (\End _{A_g} (A_g \rtimes G) , - w_g \ot 1) ) _G \\
                      &  (\oplus _g \oC ^{II} (A_g, - w_g) )_G & }} \end{equation} 
                      where:
                      \begin{itemize}
                      \item two vertical chain maps are quasi-isomorphisms as explained already;
                      \item the middle horizontal map $\tau |_{A\rtimes G}$ is induced from the composition $natural \circ \tau $;
                      \item $\psi _A$ is the quasi-isomorphic chain map defined by Baranovsky \cite[page 799]{Bar} and
                       C\u ald\u araru and Tu \cite[Proof of 6.3]{CT}; 
                      \item $\Tr$ is the generalized trace map.
                      \end{itemize}

\begin{lemma}\label{lem: key hkr}
 The triangle in  \eqref{eqn: big diag} is commutative. Hence  $\tau |_{A\rtimes G} $ is a quasi-isomorphism. 
\end{lemma}           
            
\begin{proof}
This will be clear since $g$ acts on $A_g$ trivially.
Note that the following diagram commutes,
\begin{equation}\label{diag: key} \mbox{\footnotesize \xymatrix{ a_0 \ot g_0 [ a_1 \ot g_1 | \cdots | a_n \ot g_n] \ar@{|->}[r]^{\tau |_{A\rtimes G}\ \ \ \ } \ar@{|->}[d]_{|G| \cdot \psi _A } 
               & \bigoplus _g \oa_0 \ot  g_0 [ \oa_1 \ot g_1 | \cdots | \oa_n \ot g_n g^{-1} ] \ar@{|->}[d]_{\Tr} \\
                                 |G| \cdot \oa_0 [ g_0 (\oa_1) | \cdots | g_0 \cdots g_{n-1} (\oa_n)]  \ar@{=}[r] &  
                                 \bigoplus _{g : \  g = g_0 \cdots g_n  } |G| \cdot \oa_0 [ g_0 (\oa_1) | \cdots | g_0 \cdots g_{n-1} g^{-1} (\oa_n) ] 
                                                                }} \end{equation}
                                 where $\oa$ denotes the element in $A_g$ associated to $a$.
   Here the right bottom corner is meant to have the $g$-th component 
 \[   \left\{\begin{array}{cl} |G| \cdot \oa_0 [ g_0 \oa_1 | \cdots | g_0 \cdots g_{n-1} g^{-1} \oa_n ] & \text{ if }  g = g_0 \cdots g_n  \\
                                                                0 & \text { otherwise. } \end{array} \right. \]
                                                             \end{proof}
    
       \begin{Rmk}\label{rem:commutative}
   Note that for  {for example when } $A = k$ with a nontrivial $G$, diagram \eqref{diag: key} is not commutative without the twisting $\mathfrak{tw}$ insertion in the definition of $\tau$.
   \end{Rmk}

\subsection{Proof of Proposition \ref{prop: local}}\label{sub: pf of prop local} 
Due to Remark \ref{rmk: ind cover} and the compatibility of the map $\tau$ with \v{C}ech differentials, 
the proof follows from Lemma~\ref{lem: key hkr}.

\subsection{The role of $\can$ at the level of category}
As we have already remarked, the insertion of central automorphism $\mathfrak{tw}$ was essential. 
We would like to sketch how does it appear naturally in the computation of Hochschild invariants in order to clarify its role. 
The proof of Proposition \ref{prop: local} can be interpreted as a two-step process. 
 
The first step is purely categorical.
Suppose a $k$-linear category $\mathcal C$ carries a strict $G$-action of a finite group $G$. 
We still denote corresponding endofunctors by $g:\mathcal C \to \mathcal C, \hskip 0.2cm g\in G$.  
We also denote its category of $G$-equivariant objects by $\mathcal C_G$. 
Its object consists of a pair $(E, \phi_g^E)$ where $E$ is an object of $\mathcal C$ and $\phi_g^E$ is an isomorphism $\phi_g^E : E \simeq gE$ satisfying a cocycle condition. The morphism is defined as usual. 

There is a natural functor 
    \[ \widetilde{\phantom{a}} : \mathcal C \to \mathcal C_G\]
called \textit{linearization} which is defined on objects as 
    \[\widetilde E = \left(\bigoplus_{h\in G}hE, \phi_g^{\widetilde E}\right), \hskip 0.2cm \phi_g^{\widetilde E} : \bigoplus_{h\in G} hE = \bigoplus_{h\in G}g(hE) \simeq \bigoplus_{h \in G} (gh)E.\] 
It is not hard to see that the linearization is a both left and right adjoint to the forgetful functor. 
Its essential image generates $\mathcal C_G$ and 
    \[\Hom_{\mathcal C_G}(\widetilde E_1, \widetilde E_2) \simeq \Hom_{\mathcal C}(\widetilde E_1, \widetilde E_2)^G.\]
This fact leads to the following simple description of Hochschild homology of $\mathcal C_G$. (See \cite{Pe})
    \begin{equation}
    \label{Hochschild decomposition}
    HH_*(\mathcal C_G) \simeq (\bigoplus_{g\in G} HH_*(\mathcal C, g))^G \simeq \bigoplus_{g \in \mathrm{Conj}(G)}HH_*(\mathcal C_{C(g)}, g).
    \end{equation} 
 Here, $HH_*(\mathcal C , g)$ denotes Hochschild homology with coefficient $g$, where the endofunctor $g$ is considered as a bimodule.

Second observation is geometric. For simplicity, let $\mathcal C = D(X)$ be a dg category of coherent sheaves on a smooth affine scheme $X=\Spec(A)$ acted on by a finite group $G$. One can easily extend the discussion to the case of matrix factorizations.
Each components of \eqref{Hochschild decomposition} has a simpler description:
    \begin{equation}
    \label{geometric isomorphism}
    HH_*(D_{C(g)}(X), g) \xrightarrow[\mathrm{res}]{\sim} HH_*(D_{C(g)}(X^g), g)
    \end{equation}
Notice that the action of $g$ on $X^g$ is trivial. If $(E, \{\phi_h^E\}_{h\in G})$ is a $G$-equivariant sheaves on $X$, then $\varphi^E_g = \left(\phi_g^E\right)^{-1}$ restricts to the central automorphism $\can _{E\vert_{X_g}}$ of $E\vert_{X^g}$. 
In fact, any $C(g)$-equivariant object $(F, \{\phi^F_h\}_{h\in C(g)})$ on $X^g$ carries a distinguished automophism $\can_F = \left(\phi^F_g\right)^{-1}$. 
This assignment is viewed as a natural transformation between identity functors, or an element of zeroth Hochschild cohomology;
\[[\can] \in HH^0(D_{C(g)}(X^g)).\]
{The map $\mathfrak{tw}$ on Hochschild chains is a cap product with $[\can]$. 

Lastly, observe that $D_{C(g)}(X^g)$ is generated by $\mathcal O_{X^g}$.  Notice that $g$-action on $\mathcal O_{X^g}$ is trivial so $\can$ could be ignored. This implies
    \begin{equation}
    HH_*(A_g)^{C(g)} \xrightarrow[\mathrm{inc}]{\sim}HH_*(D_{C(g)}(X^g), g).
    \end{equation}

\subsection{Mixed complex case} In general, the map $\tau$ is not a morphism of mixed complexes. In this subsection we modify $\tau$ to get a morphism of mixed complexes. 

For $\chi \in \hat{\mu _r}$ let $q\MFdg ^{\chi} (I_{\mu_r}\cX, w)$ be the full subcategory of $q\MFdg (I_{\mu_r}\cX, w)$ consisting $\chi$-eigenobjects of $q\MFdg (I_{\mu_r}\cX, w)$.
 The map $\ftw$ restricted to the subcomplex $C ( q\MFdg ^{\chi} (I_{\mu_r}\cX, w) )$, denoted by $\ftw _{\chi}$, is nothing but the multiplication by $\chi (e^{2\pi i/r})$.
  Hence \[ \ftw_{\chi} : \MC (q\MFdg ^{\chi} (I_{\mu_r}\cX) )  \to  \MC ( q\MFdg ^{\chi} (I_{\mu_r}\cX) ) \]  is an automorphism of the mixed Hochschild complex.

Consider the composition $\tau _m$ of a sequence of morphisms of mixed complexes 
\begin{multline*}\label{eqn: tau_chi}
 \tau _m:   \oMCII (q\MFdg  (\cX, w))   \xrightarrow{pullback} \oplus _{r, \chi} \oMCII  (q\MFdg ^{\chi} (I_{\mu_r}\cX, w)) \\
\hspace{1cm} \xrightarrow{\oplus \mathfrak{tw}_{\chi} } \oplus _{r, \chi}  \oMCII (q\MFdg ^{\chi} (I_{\mu_r}\cX, w)) 
                                       \xrightarrow{nat} \oplus _{r}   \Gamma (I_{\mu_r}\cX, \oMCII (q\MF_{I_{\mu_r}\cX, w} (-) ) ) ,
 \end{multline*}
where the natural map $nat$ is defined by setting \[ nat (\sum _{\chi} a^{\chi}_0 [\overline{a}^{\chi}_1 | ... | \overline{a}^{\chi}_n ] ) 
= (U\mapsto  \sum_{\chi}  a^{\chi}_{0|_{IU}} [\overline{a}^{\chi}_{1|_{IU}} | ... | \overline{a}^{\chi}_{n|_{IU}}] ). \]

\begin{Rmk}\label{rmk: bar Tr}
 While the cochain map  \[  C ^{II} ( q\MFdg (I_{\mu_r}\cX, w) ) \xrightarrow{\Tr} \oplus _{\chi} C ^{II} ( q\MFdg ^{\chi} (I_{\mu_r}\cX, w) ) \]  is an isomorphism,
  $ \oC ^{II} ( q\MFdg (I_{\mu_r}\cX, w) ) \xrightarrow{\overline{\Tr}} \oplus _{\chi} \oC ^{II} ( q\MFdg ^{\chi} (I_{\mu_r}\cX, w) ) $  is not an isomorphism in general but a
  quasi-isomorphism from the facts that $C ^{II}\to \oC^{II}$ is a quasi-isomorphism and the above $\Tr$ is an isomorphism. \end{Rmk}

\begin{Prop}\label{prop: mixed} Suppose that $\cX$ is of form $[\Spec A/ G]$ as in  Proposition \ref{prop: local}.
The morphism $\tau _m$ in the category of mixed complexes is a quasi-isomorphism.
\end{Prop}

\begin{proof}
By the definition, we need to show that $\tau _m$ is a quasi-isomorphism between Hochschild-type chain complexes. 
Replacing $\tau$ by $\tau _m$ and $\oC ^{II}$ by $\oMC ^{II}$  in diagram \eqref{eqn: big diag} we conclude the proof.
\end{proof}

\subsection{Global case}\label{sub: global hkr}
Let $\underline{\cX}$ denote the coarse moduli space of $\cX$. For an \'etale map $V\to \underline{\cX}$ 
let $\cX_V := V \ti _{ \underline{\cX} } \cX$. 
We take the sheafification $\underline{\oMC} ^{II} (q\MFdg (\cX, w))$
(resp. $\underline{\oMC} ^{II} (q\MF _{I\cX, w})$) of the presheaf 
\[ V\mapsto \oMC ^{II} (q\MFdg (\cX _V, w)) \text{ resp. } (V\mapsto \Gamma (I\cX _V, \oMC ^{II} (q\MF_{I\cX, w} (-))) \] 
both on the \'etale site of  $\underline{\cX}$.
We take the sheaf homomorphism $\underline{\tau}_m$ induced from $\tau_m$
\[  \underline{\tau}_m :      \underline{\oMC}^{II} (q\MFdg (\cX ,w ) )  \to    \underline{\oMC}^{II} (q\MF _{I\cX, w} (-)) .    \] 
\begin{lemma}\label{lem:HKRI} Suppose that $\cX$ is smooth over $k$.  
Then the induced morphism $\RR\Gamma(\underline{\tau}_m)$ fits into a diagram of isomorphisms in the derived category of mixed complexes:
\begin{equation} \label{eqn: combined maps}
\xymatrix{ \oMC (\MFdg (\cX , w)) \ar[d]_{\sim}^{ } &     \RR\Gamma ( \underline{\oMC}^{II}  ( \cO _{I\cX} , -w|_{I\cX})) \ar[d]^{\sim} \\
 \RR\Gamma \underline{\oMC}^{II}  (q\MFdg (\cX, w) ) \ar[r]^{\sim}_{\RR\Gamma (\underline{\tau}_m)}  &  \RR\Gamma \underline{\oMC}^{II}  
 (q\MF_{I\cX, w} (-)) .
}  \end{equation} 
\end{lemma}

\begin{proof}
The right vertical map is a quasi-isomorphism by the quasi-Morita invariance and the fact that for each \'etale morphism $U\to I\cX$ 
the Yoneda embedding $(\cO _{I\cX} (U) , -w) \to  q\MF _{I\cX, w} (U)$ is a
pseudo-equivalence; see \cite[Proposition 3.25]{BW} and \cite{PP}.
It remains to show that the left vertical map is a quasi-isomorphism. Let $\pi : \cX \to \underline{\cX}$ be the coarse moduli space.
By \cite[Corollary 4.6]{H-LP-Arxiv}, the presheaf $(V\mapsto \oC (\MFdg ( V \ti _{\underline{\cX}} \cX , w))$ is a sheaf on the \'etale site of $\underline{\cX}$.
It is thus enough to show that the left vertical map is a quasi-isomorphism when $\cX = [X/G]$ for a smooth variety $X$ and a finite group $G$ which follows from \cite[Theorem 6.9]{CKK1}.
\end{proof}

\begin{Thm}\label{thm: HKR and Ch first}  Suppose that $\cX$ is smooth over $k$.
Then the isomorphism 
\begin{equation}\label{eqn: global local}  \oMC  (\MFdg (\cX , w)) \cong  \RR\Gamma ( \underline{\oMC}^{II}  ( \cO _{I\cX} , -w|_{I\cX})) \end{equation}
in the derived category of mixed complexes induces the isomorphism \[ \oMC (\MFdg (\cX , w)) \cong \RR\Gamma (\Omega ^{\bullet}_{I\cX}, -dw|_{I\cX}, ud) .\]
\end{Thm}

\begin{proof}

The proof follows from Lemma \ref{lem:HKRI}  and the HKR-type isomorphism (\cite{CT, Se}) for affine orbifolds.
\end{proof}

\section{Chern character formulae}\label{sec: Chern char form}

Let $\cX$ be a smooth separated finite-type DM stack over $k$ and let $w : \cX \to \mathbb{A}^1_k$ be an algebraic function on $\cX$
with only critical value $0$.

\subsection{A formula via \vC ech model and Chern-Weil theory}
We fix an affine \'etale surjective morphism $\bp : \fU \to \cX$ from a $k$-scheme $\fU$ as in \S~\ref{sub: Cech model}.
Since $\fU$ is  an affine scheme over $k$, every $E|_{\fU}$ has a connection
\[ \nabla _{E|_{\fU}} : E|_{\fU} \to E|_{\fU} \ot \Omega _{\fU}^1 . \]
Define a connection \[ \nabla _{E |_{\fU^r}}  : E|_{\fU^r}   \to  E|_{\fU^r}  \ot \Omega _{\fU^r}^1 \] 
by letting $\nabla _{E |_{\fU^r}}  =  p_1^* \nabla _{E |_{\fU}}$, where $p_1$ is the first projection $\fU^r \to \fU$.
This gives rise to a connection 
\[ \nabla _E :  \vC  (E) \to \Omega ^1_{\cX} \ot \vC (E)    \]
where $\vC (E) := ( \bigoplus _{r \ge 0} \vC ^r  (E), \dC )$ and  $\vC ^r  (E)= \bp_{r*} \bp_{r}^{*} E $. For every $E$, fix such a connection once and for all. 

Let $I\fU$ denote the affine scheme $\fU \ti _{\cX} I\cX$. Using this affine covering of $I\cX$, we have the \v{C}ech resolution 
$\vC (E|_{I\cX})$ and the connection 
\[ \nabla _{ E|_{I\cX}} : \vC  (E|_{I\cX} ) \to \Omega ^1_{I\cX} \ot  \vC  (E|_{I\cX}) .\]
In general, for every vector bundle $F$ on $I\cX$, we can choose a connection
$\nabla _F :  \vC  (F ) \to\Omega ^1_{I\cX} \ot  \vC  (F) $.

 For each $P\in q\MFdg  (I\cX, w) )$, choose a connection $\nabla _{P}$ as above once and for all. 
 Let $R = u \nabla _{\nabla}^2 + [\nabla _{P} , \delta _{\vC (P)} ]$ a kind of the total curvature of $\nabla _P$.
By a straightforward generalization of the definition of a chain map $\tr _{\nabla}$ in \cite{CKK1} to  the stacky case, we obtain
a $k[[u]]$-linear map 
\[       \tr _{\nabla, I\cX} :                   C (q\MFdg  (I\cX, w) ) [[u]] \to           \Gamma (   \vC (\Omega ^{\bullet}_{I\cX} ) )  [[u]] \]
mapping $a_0[a_1| \cdots | a_n]$ for $a_i \in \End _{\vC (P)} (P)$, $P\in q\MFdg  (I\cX, w) )$  to 
\[  \sum _{(j_0, ..., j_n) : j_i \in \ZZ _{\ge 0}} \frac{(-1)^{|j_0 + \cdots +  j_n| }\tr (a_0 R^{j_0} [\nabla _{P}, a_1] R^{j_1} [\nabla _{P}, a_2] \cdots [\nabla _{P}, a_n] R^{j_n} ) }
{(n+|j_0 + \cdots +  j_n|)!}  . \]
It is clear how to map an arbitrary element of  $C (q\MFdg  (I\cX, w) ) [[u]]$.
 By the same proof in \cite[Appendix B]{CKK1}     the map   $\tr _{\nabla, I\cX} $ is a chain map.
Likewise we have a chain map     \[  \tr_{\nabla}  :    \underline{\oC}^{II}  (q\MF_{I\cX, w} ) [[u]]  \to p_*\vC (\Omega ^{\bullet}_{I\cX} )  [[u]]   ,\]
where $p : I\cX \to \underline{\cX}$ is the natural morphism. 

Consider a diagram of chain maps in negative cyclic type complexes 
\begin{equation}
\xymatrix{ C (q\MFdg  (\cX, w) ) [[u]] \ar[d] \ar[r]^{ \mathfrak{tw} \circ \rho _{\cX}^*} &   C (q\MFdg  (I\cX, w) ) [[u]]   \ar[r]^{\tr_{\nabla, I\cX}} &           \Gamma (   \vC (\Omega ^{\bullet}_{I\cX} ) )  [[u]] \ar[d] \\
 \RR\Gamma \underline{\oC}^{II}  (q\MFdg (\cX, w) ) [[u]] \ar[r]^{\sim}_{\ \RR\Gamma (\underline{\tau}_m)} &      \RR\Gamma \underline{\oC}^{II}  (q\MF_{I\cX, w} ) [[u]]    \ar[r]_{\ \ \ \RR\Gamma (\tr_{\nabla} )  }  &   
  \RR\Gamma (\vC (\Omega ^{\bullet}_{I\cX} ) ) [[u]]   .
}
\end{equation} Note that the diagram is commutative and all arrows possibly but two top horizontal arrows are quasi-isomorphisms.
Hence we have the following corollaries.

\begin{Cor} The chain map
\[  \tr_{\nabla, I\cX} \circ  \mathfrak{tw} \circ \rho _{\cX}^*: C (q\MFdg  (\cX, w) ) [[u]]  \to     \Gamma (   \vC (\Omega ^{\bullet}_{I\cX} ) ) [[u]]  \]
is a quasi-isomorphism.
\end{Cor}

\begin{Cor}\label{cor: ch hn form}
Under the isomorphism in \eqref{eqn: combined maps}, the Chern character $\ch _{HN} (P)$ of $P \in \MFdg (\cX, w)$
 is the class represented by \v{C}ech cocycle 
\begin{equation*}\label{eqn: stack HH ch form} 
   \tr \left(    \can _{P |_{I\cX}}  \circ \exp ( - u \nabla _{P |_{I\cX}}^2  - [\nabla_{P |_{I\cX}} , \delta _{P|_{I\cX}}  + \dC ] ) \right)  \end{equation*}
in  $\check{H} (I\fU, (\Omega ^{\bullet}_{I\cX}, -dw|_{I\cX})) $.  
\end{Cor}

\medskip

\begin{EG}\label{EG: ch} Consider a DM stack $\cX$ of the form $[X/G]$ with $X$ quasi-projective and $G$ a linearly reductive group.
Then there is a finite collection $\fU=\{ U_i \}_{i \in I}$ of  $G$-invariant affine open subset $U_i$ of $X$ such that $\bigcup _i U_i = X$. On the other hand
there is a finite subset $S$ of $G$ such that 
\[ I\cX = \sqcup _{g \in S} [ X^g / \rC _G (g ) ]. \]
Note that $I\fU = \{ U_i ^g : i \in I, g \in S \} $. Instead of affine \'etale covering we may use the affine smooth covering $\fU$ 
for a \v{C}ech model of $q\MFdg (\cX, w)$ and the chain map $\tr _{\nabla , I\cX}$.
Since $G$ is linearly reductive, each $P|_{U_i}$ has a $G$-equivariant connection, which in turn gives 
a $\rC _{G}(g)$-equivariant connection $\nabla _{i, g}$ on $P|_{U^g_i}$. This is because of the surjection of the canonical map $\Hom ^G_{\cO _{X}} (E, J(E)) \to \Hom ^G_{\cO _{X}} (E, E)$ induced from the jet sequence 
\[ 0 \to \Omega ^1_X \ot_{\cO_X} E \to J(E) \to E \to 0 . \] 
We have
\[ \ch_{HN} (P) =   \oplus _{g\in S} \tr \left(    g  \circ \exp ( - u \Pi _{i \in I} \nabla _{i, g}^2  - [\Pi _{i\in I}  \nabla_{i, g}, \delta _{P|_{U^g}} + \dC  ] ) \right) . \]
When $X$ itself is affine, then the formula simplifies to
\[ \ch_{HN} (P) =   \oplus _{g\in S} \tr \left(    g  \circ \exp ( - u  \nabla _{g}^2  - [ \nabla_{ g}, \delta _{P|_{U^g}}  ] ) \right) , \]
taking into account the fact that $[\nabla _{g}, \dC ] = 0$. 
\end{EG}

Let  $a \in \oplus _i \RR ^i\End (P)$, then it determines a class in $H^* (I\cX , (\Omega ^{\bullet}_{I\cX}, -dw|_{I\cX} ))$ under the HKR isomorphism.
Denote the class by  $\tau (a)$. The assignment $a \mapsto \tau (a)$ is called the boundary-bulk map.

\begin{Cor}       { (The boundary-bulk map formula)} 
\[ \tau (a) =  \tr \left(  a \circ   \can _{P |_{I\cX}}  \circ \exp (  - [\nabla_{P |_{I\cX}} , \delta _{P|_{I\cX}}  + \dC ] ) \right) . \]
\end{Cor}

\subsection{A formula via Atiyah class}\label{sub: Atiyah}
Let $P \in \MFdg (\cX, w)$ and let  
  \begin{equation*} 
\Omega ^{-dw} _{\cX} := [ \xymatrix{  \ar@/^.2pc/[r]^{0} \Omega ^1 _{\cX} & \cO _{\cX} \ar@/^.2pc/[l]^{-dw\wedge}    } ]  \end{equation*}
  be the matrix factorization for $(\cX, 0)$ located at amplitude $[-1, 0 ]$.
The Atiyah class $\hat{\at} (P)$ defined  in \cite[Appendix B]{FK} is 
a suitable element of \[ \Ext ^1 (P,  P \ot  \Omega ^{-dw} _{\cX} ) . \]

When $w=0$, 
we have the decomposition 
\[ \Ext ^1 (P,  P \ot  \Omega ^{-dw} _{\cX} ) =  \Ext ^0 (P, P) \oplus \Ext ^1 (P, P\ot \Omega ^1_{\cX})  \] 
and let  \[ \at (P) = proj \circ  \hat{\at} (P) \in    \Ext ^1 (P, P\ot \Omega ^1_{\cX}) ,\] where $proj$ is the projection.
For example when $P$ is a vector bundle $F$ and $\cX$ is non-stacky, then $\at (F)$ is the usual Atiyah class \cite{At}.
In this case $\hat{\at} (F) = 1_F + \at (F)$.

\begin{Def} Taking into account the convention of  the exponential $\underline{\exp}$ of $\hat{P}$ as explained in \cite{FK, KP},
 we define  a naive Chern character of $P$ by
\[ \ch (P) := \tr  \big( \underline{\exp} (\hat{\at} (P)) \big) \in H^* (\cX, ( \Omega ^{\bullet}_{\cX}, -dw)) . \]
For simplicity we abuse notation writing $\exp (\hat{\at} (P))$ for $\underline{\exp} (\hat{\at} (P))$.
\end{Def}
The correct formula for $\ch _{HH} (P)$  in \cite{KP2} is 
\begin{align}   \label{eqn: ch at form}    \ch _{HH} (P)  & = \tr \left(  \can_{P|_{I\cX}} \circ \exp (\hat{\at} (P|_{I\cX}) ) \right) \\
                                  \nonumber           & = \ch (P) + \text{ twisted part } .     \end{align}
We note that this formula  agrees with the formula in Corollary  \ref{cor: ch hn form}, since Atiyah class  $\hat{\at} (P|_{I\cX})$ is representable as
\[      (\mathrm{id}_P, - [\nabla_{P |_{I\cX}} , \delta _{P |_{I\cX}}  + \dC ]   )  \in \Gamma (I\cX, \End (P|_{I\cX}) \ot \Omega ^{-dw} _{\cX}\ot \vC (\cO _{IX}))    \]
 in \vC ech cohomology $\check{H} (I\fU, \Omega ^{\bullet}_{I\cX}, -dw|_{I\cX}) $; see \cite[Proposition 1.3]{KP}
 and 
 \[ \underline{\exp}  (\mathrm{id}_P, - [\nabla_{P |_{I\cX}} , \delta _{P |_{I\cX}}  + \dC ]   )  = \exp  (- [\nabla_{P |_{I\cX}} , \delta _{P |_{I\cX}}  + \dC  ]) . \]
The boundary bulk map formula can also be written in terms of the Atiyah class:
\[        \tau (a) =   \tr \left( a \circ \can_{P|_{I\cX}} \circ \exp (\hat{\at} (P|_{I\cX})  ) \right) .  \]

\begin{Def} For a vector bundle $E$ on $I\cX$ we define 
\[ \ch _{tw} (E) := \tr (\can _E \exp (\at (E))  \]
and Todd class $\td _{tw} (E)$ of $E$ by the expression of Todd class in terms of the Chern character $\ch _{tw} (E)$. 
For example, $\td _{tw} (T_{I\cX})$ is defined. Since $T_{I\cX}$ is fixed under the canonical automorphism, we simply
write $\td (T_{I\cX})$ for $\td _{tw} (T_{I\cX})$.
\end{Def}

\subsection{Proof of Theorem \ref{thm: HKR and Ch}}
The first statement of Theorem \ref{thm: HKR and Ch} is Theorem \ref{thm: HKR and Ch first}.
The second statement follows from \eqref{eqn: ch at form}.

\subsection{Compactly supported case} Let $Z$ be a closed substack of $\cX$ proper over $k$.
Let $P$ be a matrix factorization for $(\cX, w)$ which is coacyclic over $\cX - Z$.
Note that \[ \hat{\at} (P|_{I\cX}) \in  \Ext ^1 (P|_{I\cX},  P|_{I\cX} \ot  \Omega ^{-dw|_{I\cX}} _{I\cX} ) = \Ext ^1 _{IZ} (P|_{I\cX},  P|_{I\cX} \ot  \Omega ^{-dw|_{I\cX}} _{I\cX} ) . \]
To emphasize that $\hat{\at} (P|_{I\cX}) $ can be considered as an $IZ$-supported extension class  write $\hat{\at}_Z (P|_{I\cX}) $ for $\hat{\at} (P|_{I\cX}) $.
Let $\MFdg (\cX, w))_Z$ be the full subcategory of $\MFdg (\cX, w)$ consisting of all matrix factorization for $(\cX, w)$ that are coacyclic over $\cX - Z$.

\begin{Cor} 
There is an isomorphism 
\[ \MC (\MFdg (\cX, w))_Z \cong \RR\Gamma _Z (\Omega ^{\bullet}_{I\cX}, -dw|_{I\cX} , d)  \]
in the derived category of mixed complexes.
Under the isomorphism $\ch ^Z_{HH} (P)$ is 
equal to
\[ \tr \big( \can _{P|_{I\cX}} \exp ( \hat{\at}_Z (P |_{I\cX} )) \big)  \] 
in $\HH ^*_{IZ} ( I\cX, (\Omega ^{\bullet}_{I\cX}, -dw|_{I\cX} ) )$. 
\end{Cor}

\begin{proof}
The first statement immediately follows from Theorem~\ref{thm: HKR and Ch}. The second statement follows from a concrete chain map
for the Hochschild type chain complexes; see \cite[\S~6.2]{CKK1} for some details.
\end{proof}

\section{Stacky Riemann-Roch}

\subsection{The categorical HRR}
The results in this subsection are taken from \cite{PV: HRR, Shk: HRR} with 
a weaker condition on a dg category $\cA$. Instead assuming that $\cA$ is saturated,
we assume that $\cA$ is a locally proper and smooth.

\begin{Def} 
Let $\cA$ be a {\em locally proper} dg category: i.e., for every $x, y\in \cA$,  the dimension $\sum _{i\in \GG} \dim H^i \Hom _{\cA} (x, y)$ of total cohomology 
of $\Hom _{\cA} (x, y)$ is finite.
Let  \[ \lan , \ran _{can} :  HH _* (\cA) \ti HH_* (\cA ^{op}) \to k  \]
 be the {\em canonical pairing} of $\cA$ defined by Shklyarov. It is a $k$-linear pairing. 
\end{Def}

\subsubsection{Transformations by bimodules}\label{sub: cat grr}
 \begin{Def} For a dg category $\cC$, 
we take  the projective model structure on the  category $\Mod (\cC)$ of right $\cC$-modules. The cofibrant objects are exactly
the summands of semi-free dg-modules.
A right $\cC$-module $N$ is called {\em perfect} if $N$ is a cofibrant object which is compact in the derived category $D(\cC )$ of right $\cC$-modules.
Let $\Mod _{dg}(\cC)$ be the dg category  of right $\cC$-modules. Let $\Perf (\cC)$ be the full subcategory of $\Mod _{dg}(\cC)$
consisting of all perfect $\cC$-modules.
We call a dg category $\cC$ is {\em smooth} if the diagonal bimodule $\Delta _{\cA}$ is a perfect bimodule.
\end{Def}

From now on let $\cA$ and  $\cB$ be locally perfect and smooth dg categories unless otherwise stated.
\begin{lemma} The total dimension of Hochschild homology $HH_*(\cA)$ of $\cA$ is finite. The dg category 
$\Perf (\cA \ot \cB)$ is locally prefect and smooth.
\end{lemma}

\begin{proof}
The first claim amounts  $ \Delta _{\cA} \ot _{\cA ^{op} \ot \cA}^{\mathbb{L}} \Delta _{\cA} $ is a perfect dg $k$-module,
which follows from  tensor-hom adjunction and the conditions on $\cA$.
The second claim follows from  \cite[Lemma 2.13, 2.14, 2.15]{LS: matrix}, since $k$ is a field.
\end{proof}

For a right $\cA ^{op} \ot \cB$-module $M$ there is a dg functor $T_M : \Perf (\cA) \to \Mod_{dg} (\cB) $  sending $N$ to $N\ot _{\cA} M$. 
If $M$ is representable, then $T_M$ factors though $\Perf (\cB)$, since $\cA$ is locally proper.
Hence this is the case for every perfect $\cA ^{op} \ot \cB$-module $M$.

 Let $M \in \Perf (\cA ^{op} \ot \cB)$ and let 
$\Ch (M) = \sum_i \gamma _i \ot \gamma ^i $ under the K\"unneth isomorphism 
\[ HH_* (\Perf (\cA ^{op} \ot \cB) ) \cong HH_* (\Perf (\cA ^{op})) \ot HH_* (\Perf (\cB) ) . \]
Let  \[ HH  (T_M) : HH_* (\Perf (\cA)) \to HH_* (\Perf (\cB)) \] be the induced homomorphism from $T_M : \Perf (\cA) \to \Perf (\cB)$.

\begin{Prop}\label{prop: TM}
If  $\lan , \ran _{can}$ denotes the canonical pairing of $\Perf (\cA ^{op} \ot \cB)$, then for every $\sigma \in HH_* (\Perf (\cA))$
we have
\[ HH (T_M) (\sigma ) = \sum _i \lan \sigma, \gamma _i \ran _{can} \gamma ^i.\]
\end{Prop} 

\subsubsection{The characterization}\label{sub: char prop} 
 There are natural isomorphisms 
\[ HH_* (\Perf (\cA ^{op} \ot \cA )) \cong HH_* (\cA ^{op} \ot \cA) \cong HH_* (\cA ^{op}) \ot HH_*( \cA ) \] 
by the Morita invariance and the K\"unneth isomorphism.
Write   \[ \Ch _{HH}(\Delta _{\cA}) = \sum _i T^i \ot T_i \in HH_*(\cA ^{op}) \ot HH_* (\cA ) . \]
Then by Proposition~\ref{prop: TM} we obtain this.

\begin{Cor}\label{cor: char prop} The canonical pairing $\lan , \ran_{can}$   is characterized by two conditions:
(1) it is non-degenerate
(2) it satisfies the `diagonal decomposition' property: 
\[ \sum _i  \lan \gamma , T^i \ran  \lan T_i , \gamma ' \ran  = \lan \gamma, \gamma ' \ran  \]
for every $\gamma \in HH_*(\cA), \gamma ' \in HH_*(\cA ^{op})$.
\end{Cor}

\subsubsection{The Cardy condition}
Consider objects $x, y \in \cA$.
Let $a$ and $b$ be closed endomorphisms of $x$ and $y$, respectively.  
Let \[ L_b \circ R_a : \Hom _{\cA} (x, y) \to \Hom _{\cA} (x, y) , \ (-1)^{|a||c|} \mapsto b \circ c \circ a . \]

\begin{Thm}\label{thm: cat HRR} % {\em (Shklyarov \cite{Shk: HRR})}
We have
\[ \tr (L_b \circ R_a) = \lan [b], [a] \ran _{can}. \]
For the identities $a=1_x$, $b=1_y$, it is specialized to 
\[ \chi (x, y) = \lan \Ch_{HH} (y), \Ch_{HH} (x) \ran _{can} . \]  
\end{Thm}

\subsection{On $\MFdg (\cX, w)$}\label{sub: sm}

From this point in the text all DM stacks considered  are assumed to be quotient stacks that satisfy the resolution property. 
 
\begin{Def} An {\em LG model} is a pair $(\cX, w)$ of a smooth separated DM stack $\cX$ over $k$
and a regular function $w$ on $\cX$. We further assume that $\cX$ is a quotient stack that satisfies the resolution property.
We assume that the critical value is over $0$. If $\GG = \ZZ$, then $w=0$. 
If $\GG = \ZZ/2$, then we furthermore assume that $w: \cX \to \mathbb{A}^1$ is flat.
The pair $(\cX, w)$ will be called a {\em proper LG model} if the critical locus of $w$ is proper over $k$.
\end{Def}

Consider two proper LG models $(\cX, w)$ and  $(\cY, v)$.
We want to show that $\MFdg (\cX, w)$ is locally proper, smooth; and the following $(\dagger)$ and $(\star)$ hold:

$(\dagger)$ There is a natural dg functor  
\begin{align*} \MFdg (\cX \ti \cY, w \boxplus v) &  \to \Perf (\MFdg(\cX, w) \ot \MFdg (\cY, v)) \quad \text{ defined by }    \\
                                  E & \mapsto \Psi (E) : x \ot y \mapsto \Hom _{\MFdg (\cX \ti \cY, w \boxplus v) } (x\boxtimes y, E) .  \end{align*}
Here $w \boxplus v$ denotes $w \ot 1 + 1 \ot v$.

$(\star)$ The triangulated category $[\MFdg (\cX \ti \cY, w \boxplus v)]$ is the smallest full triangulated subcategory containing all exterior products closed under finite coproducts and
summands. Here an object of $[\MFdg (\cX \ti \cY, w \boxplus v)]$ is called an exterior product if it is isomorphic to $x\boxtimes y$ for some
$x \in \MFdg(\cX, w) , y \in \MFdg(\cY, v) $.

\begin{lemma}
\begin{enumerate}

\item  $(\star)$ implies $(\dagger)$.

\item $(\dagger)$ implies that the smoothness of $\MFdg (\cX, w)$, 

\end{enumerate}
\end{lemma}

\begin{proof}
(1) is clear. Let $\Delta : \cX \to \cX ^2$ be the diagonal morphism. Then (2) follows from that $\Psi (\Delta \cO _{\cX})$ is quasi-isomorphic to
the diagonal bimodule.
\end{proof}

Since $\MFdg (\cX, w)$ is clearly locally proper, it is enough to show $(\star)$. We check this when 
$\cX$ is a  stack quotient $[X/G]$ of a smooth variety by an action of  an affine algebraic group $G$.

\subsubsection{The case when $\GG = \ZZ$} 

Then $w=0$. Note that $(\star)$ holds true by Theorem 2.29 and Corollary 4.21 of \cite{BFK: Kernels}.

\subsubsection{The case when $\GG=\ZZ/2$}
Then $w$ is a $G$-invariant function on $X$, not identically zero on any component of $X$.
Note that $(\star)$ holds by Theorem 2.29 and Lemma 4.23 of \cite{BFK: Kernels}.

\subsection{A geometric realization of the diagonal module}\label{sub: geom real} 
Consider two proper LG models $(\cX, w), (\cY, v)$. Suppose that $\cX$, $\cY$ are stack quotients of smooth varieties by actions of affine algebraic groups.
Let $f: \cX \to \cY$ be a proper morphism with $f^* v = w$.  We call $f : (\cX, w) \to (\cY, v)$ an {\em proper LG morphism}.
Choose an affine \'etale cover $\fU \to \cX$ and $\fU ' \to \cY$. Denote  \[ \cA := \MFdg (\cX, w) , \ \cB :=  \MFdg (\cY , v )  . \] 
They are locally proper and smooth as seen in \S~\ref{sub: sm}.

Let $-w\boxplus v := - w\ot 1 + 1\ot v$ and let $ \MFdg (\cX \ti \cY , -w\boxplus v )$ be the \v{C}ech dg model of the matrix 
factorizations for $(\cX \ti \cY , -w\boxplus v )$ with respect to the affine cover $\fU \ti \fU ' \to \cX \ti \cY$.
Then by  $(\dagger)$ 
we have a natural dg functor \[ \Psi : \MFdg (\cX \ti \cY , -w\boxplus v ) \to \Perf (\cA ^{op} \ot \cB).  \]
Let $D: \cA ^{op} \to \MFdg (\cX, -w)$ be the duality functor. Then we have
a commuting diagram of isomorphisms 
\begin{equation}\label{diag: iso}  \xymatrix{ HH_* ( \MFdg (\cX \ti \cY , -w\boxplus v )) \ar[d]_{\text{hkr}} \ar[r]^(.55){HH (\Psi)} & HH_*  ( \Perf (\cA ^{op} \ot \cB)) 
\ar[d]^{\text{hkr}\circ HH(D) \ot \mathrm{id} \circ \text{k\"unneth}} \\
              \HH^{-*} (\Omega ^{\bullet}_{I\cX \ti I\cY}, d(w\boxplus - v))  \ar[r]_(.45){\text{k\"unneth}} & \HH^{-*} (\Omega ^{\bullet}_{I\cX},  dw) \ot \HH^{-*} (\Omega ^{\bullet}_{I\cY} , -dv) , } \end{equation}
where hkr and k\"unneth are the HKR type isomorphisms in \S~\ref{sub: global hkr} and the K\"unneth isomorphisms, respectively.

Consider  a  matrix factorization $K$ for  $(\cX \ti \cY, -w\boxplus v)$.
For example, we have a coherent factorization
\[ \cO _{\Gamma _f} := (\Gamma _f )_* \cO _{\cX}  \text{ for } (\cX \ti \cY, -w\boxplus v) . \] 
Since $\cX \ti \cY$ satisfies the resolution property by \cite[Theorem 2.29]{BFK: Kernels}, 
$\cO _{\Gamma _f}$ is quasi-isomorphic to  a matrix factorization.

For all $x \in \cA, y\in \cB$ there is  a natural quasi-isomorphism 
\[  \RR\Hom ( y,   q_* (  K  \ot p^*x )) \sim_{qiso}  \Hom _{\MFdg (\cX \ti \cY , -w\boxplus v)} 
( x^{\vee} \boxtimes y,     K ) \] functorial under the morphisms in the categories $\cB$ and $\cA$.
This shows the following, which will be used later.

\begin{lemma}\label{lem: geom diag}  For easy notation, write $T_{K}$ for $T_{\Psi (K)}$. Then:
\begin{enumerate} 
\item The transformation $T_{K}: \Perf (\cA) \to \Perf (\cB)$ is a dg enhancement of the Fourier-Mukai transform $[\cA] \to [\cB]$ attached to the kernel $K$.
In particular, $T_{\cO_{\Gamma _f}}$ represents $\RR f_*: [\cA] \to [\cB]$.
\item 
The bimodule $\Psi (\cO _{\Gamma _{\mathrm{id}}})$ and the diagonal bimodule $\Delta _{\cA}$ are quasi-isomorphic.
\end{enumerate}
\end{lemma} 

The second statement in the above lemma is also in Lemma 5.24 of \cite{BFK: Kernels}.

\subsection{An explicit realization of the canonical pairing}\label{sub: Pf can}

\begin{Thm}\label{thm: can pair}
Let $(\cX, w)$ be a proper LG model.
Assume that $\cX$ is a smooth quotient DM stack which satisfies the resolution property.
Then the canonical pairing coincides with the paring defined by
\begin{equation}\label{eqn: can pair form} \int _{I\cX} (-1)^{\binom{\dim _{I\cX}+1}{2}}  \cdot \wedge \cdot \wedge \frac{\td (T_{I\cX})}{\tch (\lambda _{-1} ( N^{\vee}_{I\cX / \cX}))} ,\end{equation}
 where $\dim _{I\cX}$ is the locally constant dimension function of $I\cX$.
\end{Thm}

\begin{proof}
We prove the characterization Corollary \ref{cor: char prop} for the pair \eqref{eqn: can pair form} .
The nondegeneracy follows from Serre duality \cite{Ni} as argued in \cite[\S~4.1]{FK}.
By Lemma~\ref{lem: geom diag} the `diagonal decomposition' is
 \begin{equation}\label{eqn: diag dec pf}
   \sum _i  \int _{I\cX} \gamma \cdot  t^i \cdot  \widetilde{\td} (T_{I\cX})  \int _{I\cX} t^i \cdot \gamma ' \cdot  \widetilde{\td} (T_{I\cX})           
 =   \int _{I\cX} (-1)^{\binom{\dim _{I\cX}+1}{2}}  \gamma '' \cdot  \widetilde{\td} (T_{I\cX}),
 \end{equation}
 where  \begin{align} 
  \sum_i t^i \ot t_i  & = \ch _{HH} (\Delta _* \cO _{\cX}) \in \HH ^* (I\cX, (\Omega ^{\bullet}_{I\cX}, -dw|_{I\cX})) \ot \HH^* (I\cX, (\Omega ^{\bullet}_{I\cX}, dw|_{I\cX}))  \nonumber \\
    \gamma &  \in \HH ^* (I\cX, (\Omega ^{\bullet}_{I\cX}, dw|_{I\cX})), \gamma ' \in \HH^* (I\cX, (\Omega ^{\bullet}_{I\cX}, -dw|_{I\cX})) \nonumber \\
 &  \label{eqn: alpha} \widetilde{\td} (T_{I\cX})  := \frac{\td (T_{I\cX})}{\tch (\lambda _{-1} ( N_{I\cX/\cX}^{\vee}))}. \end{align}
To show \eqref{eqn: diag dec pf}  we use the deformation to the normal cone.

Over $\PP^1$ there is a deformation stack $\cM^{\circ}$ to the normal cone $N_{\cX / \cX ^2}$: the general fiber is $\cX ^2$ and the special fiber, say over $\infty$, is 
the vector bundle stack $N_{\cX / \cX ^2} \cong T_{\cX}$. It comes with a natural morphism $h: \cM ^{\circ} \to  \cX ^2$,
 a flat morphism $pr: \cM ^{\circ} \to \PP ^1$, and a morphism
 $\widetilde{\Delta} : \cX  \ti \PP ^1 \to \cM^{\circ}$ such that $(h, pr) \circ \widetilde{\Delta}  = \Delta \ti \mathrm{id}_{\PP ^1}$.
Consider the fiber square diagram
\[ \xymatrix{ \cX \ti 0 \ar[r] \ar[d]_{\Delta} & \cX \ti \PP ^1 \ar[d]_{\widetilde{\Delta}} & \ar[l] \cX \ti \infty \ar[d]_{\delta} \\
        \cX ^2 \ar[r]       & \cM ^\circ        & \ar[l] T_\cX \\
        I\cX ^2 \ar[r] \ar[u]^{\rho _{\cX ^2}} &  \ar[u]^{\rho _{\cM ^{\circ}}} I\cM ^{\circ} &\ar[u]^{\rho_{T_{\cX}}}  \ar[l] IT_{\cX}. } \]
Here we use facts that  $I\cX ^2 \cong I \cX \ti I \cX$ and $T_{I\cX}  \overset{}{\cong}  IT_{\cX}$ by  Lemma \ref{lem: IT TI}.

Let  \[ \pi _{\cX}  : T_{\cX} \to \cX \text{ and } \pi _{I\cX} :  T_{I\cX} \to I\cX \]  be the projections from vector bundles.
Then LHS of \eqref{eqn: diag dec pf} becomes
\begin{equation}\label{1} \int _{N_{I\cX/ I\cX ^2 }} \pi_{I\cX}^* (\gamma '') (\tch (\LL \rho ^*_{T_\cX} \delta_* \cO _{\cX})) \cdot   \pi_{I\cX}^* (\widetilde{\td}_{I\cX} )^2 ; \end{equation}
by the Tor independence of the pair $(\cX \ti \PP ^1, \cM ^{\circ} \ti _{\PP ^1} p)$ over $\cM ^{\circ}$ for $p=0, \infty$ and the base change II in \S~\ref{subsub: basic prop}; 
for details see the proof of \cite[\S~3.3]{Kim}.
Let $\sigma $ be the diagonal section of the vector bundle $\pi _{\cX}^* T_\cX$ on $T_{\cX}$ and let $\Kos (\sigma)$ denote the Koszul complex
associated to the section $\sigma$. Then \eqref{1} becomes
\begin{equation*} \int _{N_{I\cX/ I\cX ^2 }} \pi_{I\cX}^* (\gamma '')    (\tch ( \rho ^*  _{T_{\cX}}\mathrm{Kos} (\sigma))) \cdot   \pi_{I\cX}^* (\widetilde{\td}_{I\cX})^2,  \end{equation*}
which equals, by the functoriality and the projection formula \S~\ref{subsub: basic prop},
\begin{equation}\label{2} \int _{I\cX }  \left( ( \gamma '' \cdot \widetilde{\td} (T_{I\cX})   \int _{\pi_{I\cX}} (\tch ( \rho ^*_{T_{\cX}}  \mathrm{Kos} (\sigma))) \cdot   \pi_{I\cX}^* (\widetilde{\td}_{I\cX} ) \right) . \end{equation}
Let $I\sigma$ be the diagonal section of the vector bundle $\pi_{I\cX} ^* T_{I\cX} $ on $T_{I\cX}$.
From the short exact sequence in \S~\ref{subsub: ses rho TX},
we have a short exact sequence
\begin{align} \label{eqn: ses 2}  0 \to \pi_{I\cX} ^*  T_{I\cX} \xrightarrow{\iota} \pi_{I\cX} ^*(  T_{\cX} |_{I\cX}) \to \pi_{I\cX} ^*  N_{I\cX / \cX} \to 0 ; \\
\nonumber  \text{    with }  \iota (I\sigma) = \pi_{I\cX}^* \sigma        \end{align}
and an equality
\begin{equation}\label{eqn: fixed part of T} T_{\cX} |_{I\cX} ^{\rm{fixed}} = T_{I\cX} . \end{equation}
Then \eqref{2} becomes, by \eqref{eqn: ses 2} \& \eqref{eqn: fixed part of T}, 
\begin{equation*} \int _{I\cX}   \left( (  \gamma '' \cdot  \widetilde{\td} (T_{I\cX}) )   \int _{\pi_{I\cX}} \ch  (\mathrm{Kos} (I\sigma) ) \pi_{I\cX}^* \td (T_{I\cX}) \right)   \end{equation*}
which becomes, by \S~\ref{subsub: comp}, 
\begin{equation*} \int _{I\cX} (-1)^{\binom{\dim _{I\cX}+1}{2}}  \gamma '' \cdot  \widetilde{\td} (T_{I\cX}). \end{equation*}
This completes the proof. 
\end{proof}

\subsection{Proof of Theorem \ref{thm: HRR}}
This follows from Theorems \ref{thm: cat HRR} and  \ref{thm: can pair}.

\subsection{GRR}\label{sub: GRR}
Consider a proper morphism $f: \cX \to \cY$ with $f^* v = w$ as in \S~\ref{sub: geom real}.
Let $K_0 (\cA)$ be the Grothendieck group of the homotopy category of a pretriangulated dg category $\cA$.
Denote $f_! : K_0 (\MFdg (\cX, w) ) \to K_0 (\MFdg (\cY , v) ) $ be the homomorphism in the Grothendieck groups 
induced from $\RR f_*$. 

\begin{Thm} \label{Thm: GRR body} {\em (=Theorem~\ref{Thm: GRR})}
The diagram 
\begin{equation*} \xymatrix{ K_0 (\MFdg (\cX, w) ) \ar[d]_{\Ch_{HH}} \ar[rrr]^{f_{!}} & & & K_0 (\MFdg (\cY , v )) \ar[d]^{\Ch_{HH}} \\ 
\ar[d]_{I_{HKR}}  HH_* (\MFdg (\cX, w)) \ar[rrr]^{HH(\RR f_*)}  & & & HH_* (\MFdg (\cY , v ) )  \ar[d]^{I_{HKR}}  \\
                                            \HH ^{-*} (I\cX, (\Omega ^{\bullet}_{I\cX}, -dw|_{I\cX})) \ar[rrr]_{  \int _{If}  (-1)^{\dim_{If}} \cdot \wedge \widetilde{\td} (T_{If}) } & & 
                                            &  \HH ^{-*} (I\cY, (\Omega ^{\bullet}_{I\cY}, -dv|_{I\cY})) } \end{equation*}
is commutative.  Here $\widetilde{\td} (T_{If}) : = \widetilde{\td} (T_{I\cX})  / If^* \widetilde{\td} (T_{I\cY})$ and $\dim_{If} = \dim _{I\cX}  -  \dim _{I\cY} $,
where $\widetilde{\td} (T_{?})$ is $\widetilde{\td}$ for $T_?$ in \eqref{eqn: alpha}. 
\end{Thm}

\begin{proof} The proof is parallel to that of Theorem 3.6 of \cite{Kim}. The upper rectangle is clearly commutative. We show the commutativity of the lower rectangle.
For  $\gamma  \in HH_* (\MFdg (\cX, w)) $ let $\ka := I_{HKR} (\gamma )$,  $\ka ' :=  I_{HKR} (HH (\RR f_* ) (\gamma )) $ and 
let \[ \ch (\Psi ( \cO _{\Gamma _f}))   = \sum_i T^i \ot T_i \in   \HH ^{*} (I\cX, (\Omega ^{\bullet}_{I\cX}, dw|_{I\cX}))  \ot   \HH ^{*} (I\cY, (\Omega ^{\bullet}_{I\cY}, -dv|_{I\cY})) . \]
then by  Proposition \ref{prop: TM} and Theorem \ref{thm: can pair} we have 
for  $\beta \in   \HH ^{-*} (I\cY, (\Omega ^{\bullet}_{I\cY}, dv|_{I\cY}))$
\begin{equation}\label{eqn: Rf 1}
 \int _{I\cY}  \ka '   \wedge \beta \wedge \widetilde{\td} (T_{I\cY}) = \sum _i  \int _{I\cX} (-1)^{{\dim _{I\cX}+1}\choose{2}}  \ka \wedge T^i \wedge  \widetilde{\td} (T_{I\cX}) \int _{I\cY} T_i \wedge \beta \wedge  \widetilde{\td} (T_{I\cY}) . \end{equation}
 Let $\pi$ denote the projection $If^*T_{I\cY}  \to I\cX$ and  let $s$ be the diagonal section
of $\pi^*If^*T_{I\cY}$ on $If^*T_{I\cY} $. Then
the deformation argument for $\Gamma _f : \cX \to \cX \ti \cY$ as in the proof of Theorem \ref{thm: can pair} shows that
\begin{align*} &  \mathrm{RHS} \text{ of } \eqref{eqn: Rf 1}  \\
& =  \int _{I\cX\ti I\cY }  (-1)^{{\dim _{I\cX}+1}\choose{2}}  (\ka   \ot \beta) \wedge \ch (\cO ^{w\boxminus v}_{\Gamma _f }) \wedge ( \widetilde{\td} (T_{I\cX}) \ot  \widetilde{\td} (T_{I\cY})) \\
& =  \int _{f^*T_{I\cY}}(-1)^{{\dim _{I\cX}+1}\choose{2}}  \pi ^* (\ka \wedge If^* \beta  \wedge \widetilde{\td} (If^* T_{I\cY}) \wedge   \widetilde{\td} (T_{I\cX}) )  \wedge \ch (\mathrm{Kos} (s))  \\
& =  \int _{I\cX} (-1)^{{\dim _{If}+1}\choose{2}}  \ka \wedge f^* \beta \wedge  \widetilde{\td} (T_{I\cX})  =  \int _{I\cY} (\int _{If} (-1)^{{\dim _{If}+1}\choose{2}}  \ka \wedge \widetilde{\td} (T_{I\cX})) ) \wedge \beta .\end{align*} 
This completes the proof.
\end{proof}

\begin{Rmk}
We briefly discuss how the GRR for $\Delta$ would compute the canonical pairing, which shows 
some relationship between GRR and the canonical paring.

Consider the Riemann-Roch map 
 \begin{align*} \ch ^{\tau}:  K_0 (\cX , w)  & \to  \HH ^* (I\cX, (\Omega ^{\bullet}_{I\cX} , -dw)) \\
 E  &\mapsto  \ch _{HH} (E)  \widetilde{\td} (T_{I\cX}).  \end{align*}
Suppose that we have a GRR type theorem for the diagonal map $\Delta : \cX \to \cX ^2$:
\begin{equation}\label{eqn: GRR Delta}  \Delta _* \ch ^{\tau} (\cO _X) = \ch ^{\tau} (\Delta _* \cO _{\cX})  =  \frac{\td (I\cX ^2) \ch_{HH}  (\Delta _* \cO _{\cX})}{\ch_{tw} (N_{I\cX ^2/\cX ^2})} . \end{equation}
This yields a formula
\begin{align*} \ch_{HH}(\Delta _* ( \cO _{\cX}))  = \Delta _* \left( \frac{ \ch_w(N_{I\cX/ \cX})  }{ \td (I\cX) } \ch _{HH} (\cO _{\cX} )\right) 
                                 =  \Delta _*  \frac{ \ch_w(N_{I\cX/ \cX})  }{ \td (I\cX) },  \end{align*}
                                since $\ch _{HH} (\cO _{\cX} ) = 1.$
Denote $ \widetilde{\td} = \widetilde{\td} {T_{I\cX}} $. 
Then 
\begin{multline*}  \int _{I\cX} (-1)^{\binom{\dim _{I\cX}+1}{2}}  \gamma \cdot t^i  \cdot \widetilde{\td} \   \int _{I\cX} (-1)^{\binom{\dim _{I\cX}+1}{2}}  t_i \cdot \gamma'   \cdot \widetilde{\td}  
\\  = \int _{I\cX ^2} \gamma \ot \gamma '  \cdot  \Delta _*  \frac{ \ch_w(N_{I\cX/ \cX})  }{ \td (I\cX) }  \cdot  \widetilde{\td} \ot  \widetilde{\td}  
 = \int _{I\cX} (-1)^{\binom{\dim _{I\cX}+1}{2}}  \gamma \cdot \gamma ' \cdot \widetilde{\td} , \end{multline*}
which is the characterization property of the canonical pairing.
Thus \eqref{eqn: GRR Delta} implies that the canonical pairing 
is $\int _{I\cX} (-1)^{\binom{\dim _{I\cX}+1}{2}}  \cdot \wedge \cdot \wedge \widetilde{\td} (T_{I\cX})$. 

\end{Rmk}

\section{Pushforward in Hodge cohomology}\label{sec: push}

\subsection{A functor $f^!$ and its base change}
In \ref{sub: Pf can}, we implicitly use the existence of the right adjoint functor $f^!$ of $\RR f_*$ and its base change formula.
In order to simplify notations, all functors in this subsection are considered as derived functors unless stated otherwise. 
    \begin{Thm}(See \cite{Ni})
    \label{f^!}
    Let $f: \cX \to \cY$ be a proper morphism between Deligne-Mumford stacks and let $F \in D^+_c(\cX)$ and $G\in D^+(\cY)$. Then there exists a right adjoint functor $f^!$ of $f_*$ such that the composition 
    \[ f_*  \sHom_\cX (F, f^! G) \xrightarrow{\mathrm{nat}} \sHom_\cY(f_* F,  f_* f^! G) \xrightarrow{\tr_f} \sHom_\cY ( f_* F, G)\]
    is an isomorphism. 
    Here, the map $\tr_f$ is the counit of the adjoint pair $ f_* \dashv f^!$. 
    \end{Thm}
Suppose that we have a tor-independent Cartesian diagram of DM stack 
    \begin{equation}
    \label{cartesian diagram}
    \xymatrix
    { \cX' \ar[r]^v \ar[d]^g &  \cX  \ar[d]^f \\
    \cY' \ar[r]^u & \cY   }
    \end{equation}
such that $f$ is proper. Because it is tor-independent, we have a base change isomorphism 
    \begin{equation}
    \label{tor-independent base change}
    \sigma: g_*v^* \xrightarrow{\sim} u^*f_*
    \end{equation}
We also have a unit and counit map from the adjoint pair $f_*\dashv f^!$
    \begin{equation}
    \label{unit counit}
    \epsilon_f: \mathrm{Id} \to f^!f_* , \hskip 0.2cm \eta_f (= \tr_f) : f_*f^! \to \mathrm{Id}.
    \end{equation}
In this setup, we want to prove  the following base change formula for Grothendieck duality  which is well-known schemes. We prove a minor  extension for Deligne-Mumford stacks under a mild assumption on $u$.
    \begin{Prop} (See \cite{LH})
    \label{duality base change}
    Further assume that the map $u$ in \ref{cartesian diagram} is representable affine of finite-Tor dimension. 
    Then the canonical base change
    \[\beta: v^*f^! \xrightarrow{\epsilon_f} g^!g_*v^*f^! \xrightarrow{\sigma} g^!u^*f_*f^! \xrightarrow{\tr_f} g^!u^*\]
    is an isomorphism. 
    \end{Prop}
    \begin{proof}
    Observe that $u$, and hence $v$, is \textit{reflexive}, which means that if $v_*\phi$ is an isomorphism, then so is $\phi$. 
    Therefore, $\beta$ is an isomorphism if $v_*\beta: v_*v^*f^! \to v_*g^!u^*$ is an isomorphism. 
    Applying Yoneda lemma and adjunction formula for $v^*\dashv v_*$, it is enough to prove that for $F\in D^+_c(\cX), G\in D^+(\cY)$, the induced map 
    \[\beta: \Hom_{\cX'}(v^*F, v^*f^!G)\longrightarrow \Hom_{\cX'}(v^*F, g^!u^*G)\]
    is an isomorphism. 
    Consider the following commutative diagram of sheaves on $\cY'$;
    \begin{equation}
    \label{GD diagram I}
    \xymatrix
    {g_*\sHom_{\cX'}(v^*F, v^*f^!G) \ar[dr]_{\tilde \beta} \ar[r]^{g_*\beta}    &   g_*\sHom_{\cX'}(v^*F, g^!u^*G)  \ar[d]_{\sim}^{\ref{f^!}} \\
    &   \sHom_{\cY'}(g_*v^*F, u^*G)}
    \end{equation}
    The morphism $\tilde \beta$ is an instance of a sheaf version of base change morphism. It is defined as a composition
    \begin{multline}
    \tilde \beta_{(F', G)} : g_*\sHom_{\cX'}(F', v^*f^!G) \rightarrow \sHom_{\cX'}(g_*F', g_*v^*f^!G)\\
    \xrightarrow{\sigma} \sHom_{\cY'}(g_*F,, u^*f_*f^!G) \xrightarrow{\tr_f} \sHom_{\cY'}(g_*F,' u^*G).
    \end{multline}
    for each $F'\in D^+_c(\cY), G\in D^+(\cY)$. Then $\tilde \beta = \tilde \beta_{(v^*F, G)}$. One can easily check that it is an isomorphism from the following commutative diagram;
    \begin{equation}
    \label{GD diagram II}
    \xymatrix
    {g_*\sHom_{\cX'}(v^*F, v^*f^!G) \ar[d]_{\sim} \ar[r]^{ \tilde \beta}    &   \sHom_{\cY'}(g_*v^*F, u^*G)  \ar[d]_{\sim}^{\sigma} \\
    g_*v^*\sHom_\cX(F, f^!G) \ar[d]_{\sim}^{\sigma}     &   \sHom_{\cY'}(u^*f_*F, u^*G) \ar[d]_{\sim}\\
    u^*f_*\sHom_\cX(F, f^!G) \ar[r]_{\sim}^{\ref{f^!}}      &   u^*\sHom_{\cY}(f_*F, G).}
    \end{equation}
    Here, the vertical isomorphisms on the top left and bottom right follows from the fact that the natural morphism of sheaves
    \[u^*\sHom_{\cY}(F, G)\to \Hom_{\cY'}(u^*F, u^*G)\]
    is an isomorphism because $u$ (and hence, $v$) is of finite tor-dimension. 
    \end{proof}
We apply \ref{duality base change} to the deformation to the normal cone. 

\subsubsection{Some basic properties}\label{subsub: basic prop}
 In this section all stacks are assumed to be smooth separated DM stacks of finite type over $k$.
Let $f: \cX \to \cY$ be a morphism. Assume that they are pure dimensional and let $d$ be $\dim \cX - \dim \cY$.

\begin{Def}
Once we have the right adjoint functor $f^!$ of $\RR f_*$, as in \cite{Kim} we can define 
\[ \int _f: H ^{q}_{\sharp _1} (\cX, \Omega ^p_{\cX}) \to H^{q-d}_{\sharp _2} (\cY, \Omega _{\cY}^{p-d}) \]
where $(\sharp_1, \sharp_2)$ is either $(c, c)$ or $(cf, \emptyset)$. 
When $\cY$ is  $\Spec\,k$, write $\int _{\cX}$ for $\int _f$.
\end{Def}

\begin{Rmk} In the construction of $\int_f$, Nagata's compactification and the resolution of singularities were used. In our separated DM stack setup both are known by
\cite{Rydh} and \cite{Tem}, respectively.  
\end{Rmk}

The following can be straightforwardly proven as in \cite[\S~3.6]{Kim}.
\begin{enumerate}
\item (Base change I) Consider a Cartesian diagram \eqref{cartesian diagram}. Assume that $f$ is a flat, proper and l.c.i morphism. 
Then 
\begin{equation*}\label{base change I}
\int _{\cX '} v^* (\gamma ) = u^* (\int _f \gamma ). \end{equation*}

\item (Base change II) Consider a Cartesian diagram \eqref{cartesian diagram}. 
Assume that $f$ is a flat morphism, $\cY$ is a connected $1$-dimensional smooth scheme,
and $u$ is the embedding of a closed point $\cY'$ of $\cY$.
Then 
\begin{equation*}\label{base change II}
\int _{\cX '} v^* (\gamma ) = u^* (\int _f \gamma ) \in k . \end{equation*}

\item (Functoriality) Let $\cX \xrightarrow{f} \cY \xrightarrow{g} \cZ$ be morphisms. 
Then \[
\int _{g \circ f} = \int _g \circ \int _f . \]

\item (Projection formula)
Let $\cX \xrightarrow{f} \cY$ be a morphisms.
Then \[ \int _f (f^* \sigma \wedge \gamma) = \sigma \wedge \int _f \gamma \]
for $\gamma \in H^d_{cf} (\cX, \Omega ^d_{\cX})$ and $\sigma \in H^q (\cY, \Omega _{\cY}^p)$.

\end{enumerate}

\subsubsection{Computation}\label{subsub: comp}  
Let $E$ be a vector bundle on $\cX$ of rank $n$, let $\pi : E \to \cX$ be the projection, and let $s$ be the diagonal section of $\pi ^*E$.
Since $\pi$ is representable, 
we have \[ \int _{\pi} \ch (\Kos (s)) \td (\pi ^* E) = (-1)^{{n+1}\choose{2}} \]
by the base change I in \S~\ref{subsub: basic prop} and the computation of \cite[\S~3.6.6]{Kim}.

\end{document}